\documentclass[12pt]{amsart}
\usepackage{amsmath}
\usepackage{amssymb}
\usepackage[all]{xy}
\theoremstyle{plain}
\newtheorem{thm}{Theorem}[section]
\newtheorem{thm*}{Theorem}[section]
\newtheorem{cor}[thm]{Corollary}
\newtheorem{prop}[thm]{Proposition}
\newtheorem{lemma}[thm]{Lemma}

\newtheorem{conj}[thm]{Conjecture}
\newtheorem{ques}[thm]{Question}

\theoremstyle{definition}
\newtheorem{defn}[thm]{Definition}
\newtheorem{remark}[thm]{Remark}
\newtheorem{ex}[thm]{Example}

\newcommand{\bG}{\mathbb G}

\newcommand{\Coind}{\rm Coind}

\newcommand{\bZ}{\mathbb Z}
\newcommand{\bP}{\mathbb P}

\newcommand{\bu}{{\bullet}}
\newcommand{\ol}{\overline}
\newcommand{\ul}{\underline}

\newcommand{\Gas}{\mathbb G_{a(s)}}

\newcommand{\Spec}{\text{Spec}\,}

\newcommand{\Aut}{\text{Aut}}
\newcommand{\Ext}{\text{Ext}}
\newcommand{\End}{\text{End}}
\newcommand{\Ker}{\text{Ker}\,}
\newcommand{\Hom}{\text{Hom}}
\newcommand{\Ind}{\text{Ind}}

\newcommand{\Proj}{\text{Proj}\,}

\newcommand{\Z}{\mathbb Z}
\newcommand{\bN}{\mathbb N}

\newcommand{\rk}{\text{Rk}\,}
\newcommand{\Id}{\text{Id}\,}

\newcommand{\N}{\mathcal N}
\newcommand{\cL}{\mathcal L}
\newcommand{\R}{\mathcal R}
\newcommand{\I}{\mathcal I}

\newcommand{\Tr}{\rm Tr \,}
\newcommand{\D}{\rm D}

\def\CE{{\mathcal{E}}}
\def\HHH{\operatorname{H}\nolimits}
\def\proj{\operatorname{(proj)}\nolimits}
\def\Rad{\operatorname{rad}\nolimits}
\def\soc{\operatorname{soc}\nolimits}
\def\JType{\operatorname{JType}\nolimits}
\def\stmod{\operatorname{stmod}\nolimits}

\def\CE{{\mathcal{E}}}
\def\HHH{\operatorname{H}\nolimits}
\def\proj{\operatorname{(proj)}\nolimits}

\def\Id{\text{Id}}
\def\Dim{\operatorname{dim}\nolimits}
\def\Hom{\operatorname{Hom}\nolimits}

\begin{document}

 \title{Modules of constant Jordan type}

\author[Jon F. Carlson, Eric M. Friedlander and 
Julia Pevtsova]{Jon F. Carlson$^*$, Eric M. Friedlander$^*$ and 
Julia Pevtsova$^*$}

\address{Department of Mathematics, University of Georgia, Athens, GA}
\email{jfc@math.uga.edu}

\address {Department of Mathematics, Northwestern University,
Evanston, IL}
\email{eric@math.northwestern.edu}

\address {Department of Mathematics, University of Washington, 
Seattle, WA}
\email{julia@math.washington.edu}

\thanks{$^*$ partially supported by the NSF }



\begin{abstract}
We introduce the class of modules of constant Jordan type 
for a finite group scheme $G$ over a field $k$ of characteristic $p > 0$.  
This class is closed under taking direct sums, tensor products, duals, 
Heller shifts and direct summands, and includes endotrivial modules.  
It contains all modules in an Auslander-Reiten component which has at 
least one module in the class.  Highly non-trivial examples are constructed 
using cohomological techniques.  We offer conjectures suggesting that there 
are strong conditions on a partition to be the Jordan type associated to a 
module of constant Jordan type.
\end{abstract}

\maketitle

\tableofcontents


\section{Introduction}
In \cite{FP1} and \cite{FP2}, the second and third
authors have introduced a seemingly naive approach to the study of
representations of finite groups and related structures on vector spaces
over a field $k$ of characteristic $p > 0$.  The basic 
idea is to restrict representations to certain subalgebras (``$\pi$-points") 
isomorphic to the
group algebra of $\bZ/p\bZ$, for we completely understand the representation
theory of the algebra $k\bZ/p\bZ$ in terms of partitions 
(or ``Jordan types").  
The simplicity of this approach enables the
authors to consider representation theory in a very general context
(of a finite group scheme $G$ over an arbitrary field of characteristic $p>0$)
and prove both global results about the stable module category and
explicit results for specific examples.  The naivety of the approach is
somewhat misleading for underlying many theorems are somewhat
difficult cohomological results, especially results giving finite generation 
and detection modulo nilpotence on subalgebras of special form.

In a recent paper \cite{FPS}, the second and third authors in 
collaboration with Andrei Suslin have adopted this naive point of view
to formulate and investigate new invariants for such representations.  
The authors introduce ``maximal" and ``generic" Jordan types for a 
given representation whose existence even in the very special case of 
the finite group $\bZ/p\bZ \times \bZ/p\bZ$ is highly non-trivial. 

Indeed, this special example $G = \bZ/p\bZ \times \bZ/p\bZ$ (for $p > 2$)
is challenging from a representation-theoretic point of view
for its group algebra has wild representation type.  With such ``wildness"
in mind, it is natural to investigate classes of representations of $G$
with special properties.  That is the purpose of this present paper, in
which we investigate modules of constant Jordan type.  Although the 
formulation of this concept requires the approach of ``$\pi$-points" 
and our study utilizes many of the techniques of the papers mentioned
above,  the  resulting class of modules appears to be a most natural 
one to study for those considering the modular representation theory
of finite groups, $p$-restricted Lie algebras, and other finite group schemes.
We remark here on two aspects of this class of modules of constant
Jordan type:  it includes the much-studied class of endotrivial modules
and also includes many other modules even in the special case of 
an elementary abelian $p$-group; the classification of such modules
of constant Jordan type appears to be very difficult, sufficiently difficult
that even in the special case $G = \bZ/p\bZ \times \bZ/p\bZ$ we 
can only speculate on what Jordan types are realized.

 In Definition \ref{def}, we introduce the concept of a $kG$-module
$M$  of constant Jordan type, a finite dimensional module with
the property that $\alpha_K^*(M_K)$ has Jordan type independent
of the  $\pi$-point $\xymatrix{\alpha_K: K[t]/t^p \ar[r] & KG}$ with
$K/k$ an arbitrary field extension.   As verified in Theorem \ref{endo},
a $kG$-module is an endotrivial module if and only if it has constant
Jordan type of a very special form.  For certain explicit finite 
group schemes,  
various examples of modules of constant Jordan type can be
constructed directly as we show in \S \ref{exam}.  Much of
our effort in the first half of this paper is dedicated to showing
that the class of modules of constant Jordan type is closed under
various  natural operations:  Heller shifts (Proposition \ref{heller}),
direct sums (Proposition \ref{heller}), direct summands
(Theorem \ref{summand}), linear duals (Proposition \ref{dual}),
and tensor products (Corollary \ref{tensor}).

        To establish these results, we continue the study initiated in
\cite{FPS} of the condition that a $\pi$-point $\alpha_K$ be maximal
for a given $kG$-module $M$.  
This analysis should be of interest for the study of general $kG$-modules. 
In \S 3, we formulate the natural
relationship of strict specialization of $\pi$-points (closely related to the
relationship of specialization of equivalence classes of $\pi$-points
considered in \cite{FP2}).  In \S 4, we investigate the surprisingly
subtle behavior of the condition of maximality of $\pi$-points
with respect to the tensor product of two given $kG$-modules.
The relevance of maximality of $\pi$-points for a $kG$-module $M$
is emphasized by Proposition \ref{empty} which asserts that the
 $kG$-module $M$ has constant Jordan type if and only if the
 non-maximal support variety of $M$, $\Gamma(G)_M$, is empty.

        In the second half of this paper, we give several methods
of constructing modules of constant Jordan type.   One is
provided in Proposition \ref{extend1} and another in Theorem
\ref{create1} (as well as Proposition \ref{create2}).  Much
preliminary effort is required for us to establish in Theorem
\ref{not-endo} that our second method provides examples which
can not be realized by the first.  Indeed, the example provided
by this theorem shows how subtle is the behavior of the class
of modules of constant Jordan type with respect to extensions.
A third method of construction using the Auslander-Reiten 
theory of almost
split sequences is detailed in \S \ref{AR}:  Theorem
\ref{component} establishes that
any module in the same Auslander-Reiten component as 
a module of constant Jordan type is also of constant Jordan type,
whereas Theorem \ref{stableone} constructs $kG$-modules
of constant Jordan type $n[1] + \proj$ provided that
$G$ satisfies a mild cohomological property.

As one indication of the combinatorics involved in the existence
of modules of constant Jordan type, we show in Theorem \ref{ranks}
that our techniques give a new proof of a special case of 
Macaulay's Generalized Principal Ideal Theorem.

We have made little progress in classifying those partitions which are
realizable as the Jordan type of modules of constant Jordan type.
For example, for a rank 2 elementary abelian $p$-group $E$ with $p > 3$, we
conjecture but can not prove  that no partition of type $n[p] + 1[2]$
is the Jordan type of $kE$-module of constant Jordan type.
In \S \ref{qc}, we mention numerous questions
and conjectures on the constraints of a Jordan type associated to
a module of constant Jordan type.

Throughout this paper, $k$ will denote an arbitrary
field of finite characteristic
and $p>0$ will denote the characteristic of $k$.  Without explicit
mention to the contrary, $kG$-modules
are assumed to be finitely generated. We let $M_n(k)$ denote 
the algebra of $n\times n$ matrices over $k$.

We thank Valery Alexeev, David Eisenbud, and Sasha Premet 
for help with Theorem~\ref{ranks}. We also thank Karin Erdmann 
for drawing our attention to possible connections
with the Auslander-Reiten theory.  We are particularly grateful to 
the referee for many constructive comments and suggestions.

\vspace{0.2in}


\section{Constant Jordan Type}
\label{ct}

In this first section, we introduce modules of constant Jordan type
and investigate some of the basic properties of this class of modules.

Recall that a finite group scheme $G$ (over $k$) is a 
group scheme over $k$ whose coordinate
algebra $k[G]$ is finite dimensional over 
$k$.  We denote the linear dual of $k[G]$ by
$kG$ and call this the group algebra of $G$.   
A (rational) $G$-module is a comodule 
for $k[G]$ or equivalently a $kG$-module.
If $K/k$ is a field extension, then we denote by
$G_K$ the base change of the $k$-group scheme
$G$ to the $K$-group scheme 
$G_K = G \times_{\Spec k} \Spec K$. We observe
that the group algebra $KG_K$ of $G_K$ equals $KG
= K\otimes_k kG$.

We remind the reader that the isomorphism 
class of a finite dimensional $k[t]/t^p$-module
$M$ of dimension $n$ (over $k$) is given by 
a partition of $n$ into blocks of size $\leq p$.
Equivalently, if we let 
$\xymatrix@-.5pc{\rho_M: k[t]/t^p \ar[r] & M_n(k)}$ 
be the representation associated to
$M$, then the isomorphism type of $M$ is specified 
by the conjugacy class of the
element $\rho_M(t) \in M_n(k)$ whose $p$-th power is $0$.   
We shall often denote the isomorphism
type of $M$ by $a_p[p] + \cdots + a_1[1]$, 
where $a_i$ denotes the number of blocks of
size $i$ in the partition of $n$ associated to $M$.

We call the isomorphism type of a finite dimensional
$k[t]/t^p$-module $M$ the {\it Jordan type} of $M$.
For any finite dimensional $k[t]/t^p$-module $M$,
the {\it stable Jordan type} of $M$ is the ``stable equivalence''
class of Jordan types, where two Jordan types
$a_p[p] + \cdots + a_1[1]$ and $b_p[p] + \cdots + b_1[1]$
are stably equivalent if $a_i = b_i,$ for all $i < p$.

We may view a Jordan type $a_p[p] + \cdots + a_1[1]$ as a 
{\it partition} of $n = \sum_{i=1}^p ia_i$.   If 
$\sum ia_i = \sum ib_i$, then we say that the Jordan type
 $\ul a = a_p[p] + \cdots + a_1[1]$ is greater or equal to
the Jordan type $\ul b = b_p[p] + \cdots +  b_1[1]$ 
(denoted $\ul a \geq \ul b$) provided that
\begin{equation}
\label{dom}
\sum_{i=j}^p ia_i \ \geq \ \sum_{i=j}^p ib_i, \quad 1\leq j \leq p.
\end{equation}
If $\ul a ~ \geq ~ \ul b$ and if $\sum_{i=j}^p ia_i > 
\sum_{i=j}^p ib_i$ for some $j$, then we write $\ul a ~ > ~ \ul b.$
Note that this is the usual dominance ordering on partitions.

\begin{remark}
Let $M, ~ N $ be  $k[t]/t^p$-modules of dimension $n$ given by 
$$\xymatrix@-.5pc{\rho_M, ~ \rho_N: k[t]/t^p \ar[r]  & M_n(k).}$$ 
Then the Jordan type $\ul a$ of $M$ is greater or equal to 
(respectively, greater than) the
Jordan type of $\ul b$ of $N$ if and only if for every $j, 1 \leq j < p$,
the rank of $\rho_M^j$ is greater or equal to the the rank of $\rho_N^j$
(resp., and strictly greater for some $j$).
\end{remark}

\begin{defn}
\label{pi}
A $\pi$-point for a finite group scheme $G$ 
is a left flat map of $K$-algebras 
$\xymatrix@-.5pc{\alpha_K: K[t]/t^p \ar[r] & KG}$, for some field 
extension $K/k$,
which factors through the group algebra 
$KC_K \subset KG_K$ of some unipotent abelian subgroup
scheme $C_K \subset G_K$.  If $M$ is a finite 
dimensional $kG$-module, the Jordan type of
the $\pi$-point $\alpha_K$ on $M$ is the 
isomorphism class of the $K[t]/t^p$-module
$\alpha_K^*(M_K)$ (where $M_K = K\otimes_k M$).
We emphasize here that $\alpha_K^*(M_K)$ denotes the restriction
of $M_K$ to a $K[t]/t^p$-module along the map $\alpha_K$.
We say that the Jordan type of $\alpha_K^*(M_K)$ 
is the Jordan type of $\alpha_K$ on $M$.
\end{defn}

\begin{defn}\label{specialize}
Let $\xymatrix@-.5pc{\alpha_K: K[t]/t^p \ar[r] & KG}$, 
$\xymatrix@-.5pc{\beta_L: L[t]/t^p \ar[r] & LG}$ 
be $\pi$-points of $G$.  Then
$\alpha_K$ is said to be a {\it specialization} 
of $\beta_L $ 
(written $\beta_L \downarrow \alpha_K$) 
if for every finite dimensional
$kG$-module $M$ the $K[t]/t^p$-module $
\alpha_K^*(M_K)$ is projective whenever
the $L[t]/t^p$-module
$\beta_L^*(M_L)$ is projective.   We say 
that $\alpha_K, \beta_L$ are {\it equivalent}
and write $\alpha_K \sim \beta_L$ provided that
$\alpha_K \downarrow \beta_L$ and
$\beta_L \downarrow \alpha_K$.
\end{defn}

The following theorem summarizes the close 
relationship between the set of equivalence
classes of $\pi$-points of $G$ and the 
cohomology $\HHH^\bu(G,k)$.  Here, $\HHH^\bu(G,k)
= \HHH^*(G,k)$, the cohomology algebra of $G$ 
provided that $p=2$; for $p > 2$, $\HHH^\bu(G,k)
\subset \HHH^*(G,k)$ denotes the commutative 
subalgebra of even dimensional classes.

\begin{thm} (\cite[3.6]{FP2})
\label{iso}
The set of equivalence classes of $\pi$-points 
of a finite group scheme $G$, written
$\Pi(G)$, admits a scheme structure determined 
by the stable module category, $stmod(kG)$.
With this structure, $\Pi(G)$ is isomorphic 
to the scheme $Proj \HHH^\bu(G,k)$.

In particular, the closed subsets of $\Pi(G)$ 
are of the form $\Pi(G)_M$ where $M$
is a finite dimensional $kG$-module and $\Pi(G)_M$ 
is the subset of those equivalence
classes of $\pi$-points 
$\xymatrix@-.5pc{\alpha_K: K[t]/t^p \ar[r] & KG}$ 
such that $\alpha_K^*(M_K)$ is not
projective.
\end{thm}

We now introduce modules of constant Jordan type, whose study is the
primary object of interest in this paper.

\begin{defn}\label{def}
The finite dimensional $kG$-module $M$ is said 
to be of constant Jordan type if the
Jordan type of $\alpha_K^*(M_K)$ is independent 
of the choice of $\pi$-point 
$\xymatrix@-.5pc{\alpha_K: K[t]/t^p \ar[r] & KG}$.
\end{defn}

\begin{remark}
\label{indep}
The Jordan type of $\alpha_K^*(M_K)$ for a finite dimensional
$kG$-module $M$ at a $\pi$-point $\xymatrix@-.5pc{
\alpha_K: K[t]/t^p \ar[r] & KG}$
typically depends not only upon the equivalence class $[\alpha_K]
\in \Pi(G)$ but also upon the representative of this equivalence
class.  However, there are some exceptions. 
The central conclusion of \cite{FPS} is that 
in either of the following two situations,
the Jordan type of $\alpha_K^*(M_K)$ does not change 
if we replace $\alpha_K$ by some $\beta_L$ with $\alpha_K \sim \beta_L$:
\begin{enumerate}
\item 
 If $[\alpha_K] \in \Pi(G)$ is a generic point; otherwise said,
if $\alpha_K$ is a generic $\pi$-point.
\item
If for the given finite dimensional $kG$-module $M$, there does 
not exist any $\pi$-point $\beta_L$ such that the Jordan type
of $\beta_L^*(M_L)$ is strictly greater than the Jordan type of
$\alpha_K^*(M_K)$. In this situation we say that  $\alpha_K$ has maximal
Jordan type on $M$.
\end{enumerate}

We recall \cite[5.1]{FPS} that the non-maximal support variety, 
$\Gamma(G)_M \subset \Pi(G)$ 
associated to a finite dimensional $kG$-module $M$ is
defined to be the (closed) subspace of those points
$x\in \Pi(G)$ with the property that for some (and thus any) representative
$\alpha_K$ of $x$ the Jordan type of $\alpha_K^*(M_K)$ is 
not maximal for $M$, or equivalently, $\alpha_K$ does not have 
maximal Jordan type on $M$.
\end{remark}

Remark \ref{indep}(2) immediately leads us to the following equivalent
formulation of the property of constant Jordan type.

\begin{prop}
\label{rep}
A finite dimensional $kG$-module is of constant Jordan type 
$a_p[p] + \cdots + a_1[1]$ if and only if for 
each equivalence class $[\alpha_K] \in \Pi(G)$ 
there exists some representative
$\xymatrix@-.5pc{\alpha_K: K[t]/t^p \ar[r] & KG}$ 
with the property $\alpha_K^*(M_K)$
has type $a_p[p] + \cdots + a_1[1]$.
\end{prop}

Since any $\pi$-point $\xymatrix@-.5pc{
\alpha_K: K[t]/t^p \ar[r] & KG}$ has the property
that $\alpha_K^*$ commutes with direct sums and (modulo
projectives) Heller shifts, we conclude the following.

\begin{prop}
\label{heller}
Let $G$ be an arbitrary finite group scheme.
\begin{itemize}
\item The trivial $kG$-module $k$ has constant Jordan type.
\item  A finite dimensional projective 
$kG$-module has constant Jordan type.  
If $kG$ is not
semi-simple, then the Jordan type of a 
$kG$-projective module $P$ is equal to $\frac{\Dim_kP}{p}[p]$.
\item If $\Omega^i(k)$ denotes the $i$-th 
Heller shift of $k$ for some $i \in \bZ$, then 
$\Omega^i(k)$ has constant Jordan type 
equal to $n[p] + 1[1]$ for some $n \geq 0$ if $i$ is
even and equal to $m[p] + 1[p-1]$ for some $m \geq 0$ if $i$ is odd.
\item
If $M$ has constant Jordan type, 
then $\Omega^i(M)$ also has constant
Jordan type for any $i \in \bZ$ 
\item If $M, \ M^\prime$ are $kG$-modules of constant Jordan type,
then $M \oplus M^\prime$ also has constant Jordan type.
\end{itemize}
\end{prop}

The preceding proposition will be supplemented in subsequent sections
by propositions asserting that other familiar operations on modules of 
constant Jordan type yield modules of constant
Jordan type:  taking a direct summand (by Theorem \ref{summand}), 
taking the tensor product (by Corollary \ref{tensor}), and taking 
$\Hom_k(-,-)$ (by \ref{const}).

We make explicit the following elementary functoriality property.

\begin{prop}
If $\xymatrix@-.5pc{f: H \ar[r] & G}$ is a flat 
map of finite group schemes and if $M$ is a 
$kG$-module of constant Jordan type, then $f^*(M)$ is a $kH$-module
of the same constant Jordan type.
\end{prop}

\begin{proof}
If $\xymatrix@-.5pc{\alpha_K: K[t]/t^p \ar[r] & KH}$ 
is a $\pi$-point of $H$, then 
$f \circ \alpha_K$ is a $\pi$-point of $G$ and the Jordan type 
of $\alpha_K^*((f^*M))_K)$ equals that
of $(f \circ \alpha_K)^*(M_K)$.
\end{proof}

\begin{remark}
\label{closed}
To verify whether or not a given finite 
dimensional $kG$-module $M$ has constant
Jordan type it suffices to check that the 
Jordan type of $\alpha_K^*(M_K)$ does not
vary as $[\alpha_K] \in \Pi(G)$ ranges 
over closed points of $\Pi(G)$.  Thus, it
suffices to consider $\alpha_K$ with $K/k$ 
finite.  In particular, if $k$ is algebraically
closed, it suffices to consider $k$-rational 
points of $\Pi(G)$.
The collection of these points is denoted
by $P(G)$, and was investigated  extensively
in \cite{FP1}.

Because a $kG$-module $M$ has constant Jordan 
type if and only if its base change
$M_K$ has constant Jordan type as a $KG$-module 
for any field extension $K/k$,
one could replace $k$ by its algebraic closure and consider
only $k$-rational points for the algebraically closed field $k$. 
\end{remark}

\vskip .2in


\section{Examples of modules of constant Jordan type}
\label{exam}

We give examples of modules of constant Jordan type in special
situations.  Perhaps it is worth remarking that these examples are quite
different from endotrivial modules considered in \S \ref{end}.

\begin{prop}
\label{aug}
Let $E$ be an elementary abelian $p$-group and let $I \subset kE$ be the
augmentation ideal of the group algebra $kE$.  Then, for any $n \geq m$,
$I^m/I^n$ is a $kE$-module of constant Jordan type.
\end{prop}

\begin{proof}
Let $\xymatrix@-.5pc{\beta: k[t]/t^p \ar[r] & kE} 
\simeq k[t_1,\ldots,t_r]/(t_1^p,\ldots,t_r^p)$ be the
$\pi$-point defined by $\beta(t) = t_1$ and let 
$\xymatrix@-.5pc{\alpha_K: K[t]/t^p \ar[r] & KE}$ be an arbitrary $\pi$-point.
Since $\alpha_K$ is flat and $\alpha_K(t)$ has $p$-th power $0$,  
$\alpha_K(t)$ is a polynomial in $t_1,\ldots,t_r$ with
constant term 0 and non-vanishing linear term (\cite[2.2,2.6]{FP1}).
Consequently, we may choose an
automorphism $\xymatrix@-.5pc{\theta_\alpha: KE \ar[r] & KE}$ 
which sends $t_1$ to $\alpha_K(t)$, so that 
$\alpha_K = \theta_\alpha \circ \beta$. 
The automorphism $\theta_{\alpha}$ necessarily sends any power $I^m$ of the 
augmentation ideal isomorphically onto itself, so that 
$$\theta_\alpha^{-1}: \alpha_K^*(I^m/I^n) ~ = ~ (\theta_{\alpha} 
\circ \beta)^*((I^m/I^n)_K) ~ \simeq ~
(\beta^*(I^m/I^n))_K.$$
In other words, the Jordan type of $\alpha_K^*(I^m/I^n)$ does not depend upon
the choice of $\pi$-point $\alpha_K$.
\end{proof}

\begin{ex}
\label{elemex}
As an elementary, specific case of Proposition \ref{aug},  
we consider $kE/I^2$ where $E$ is an elementary abelian $p$-group of rank $r$.
This is a $kE$-module of dimension $r+1$ and can be represented explicitly as
follows.  Give the $kE$-module structure on $k^{r+1}$ by 
defining the generators
$\{ t_1,t_2,\ldots,t_r \}$ to act  by multiplication by 
$\{ e_{1,2}, e_{1,3}, \ldots,e_{1,r+1} \}$, 
pairwise commuting elementary matrices of size 
$(r+1) \times (r+1)$ with $p$-th power $0$.  The
constant Jordan type of the module $kE/I^2$ is $1[2] + (r-1)[1]$.
\end{ex}

\begin{remark}
By Proposition \ref{dual}, the dual $(kE/I^2)^\#$ of $kE/I^2$ 
is also a module of
constant Jordan type of the same Jordan type as $kE/I^2$.  In the special case
of $r = 2$, $(kE/I^2, (kE/I^2)^\#)$ constitute the example 
produced years ago by Jens Jantzen of two
non-isomorphic modules with the same ``local Jordan type."
\end{remark}

\begin{ex}
We give a somewhat more interesting example to show how the 
residue characteristic
$p$ plays a role.  Consider the elementary abelian $p$-group $E$ of
rank 2, so that $kE = k[x,y]/(x^p,y^p)$.  Define the $kE$-module $W$ of dimension 13 
generated by $v_1,v_2, v_3,v_4$ and spanned as a $k$-vector space by
$$\{  v_1,v_2,v_3,v_4,x(v_1),x(v_2),x(v_3),x(v_4),x^2(v_1),x^2(v_2),x^2(v_3),
y(v_1),yx(v_1) \}$$ with 
$$x(v_i) = y(v_{i+1}),~ y^2(v_1) = x^2(v_4) = x^3(v_i) = 0.$$

We represent this module with the following diagram
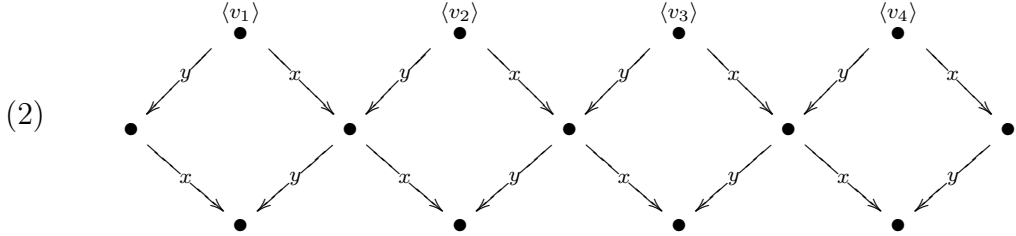
\begin{equation}
\label{dia5}
\begin{xy}*!C\xybox{%
\xymatrix{ &\stackrel{\langle v_1 \rangle }{\bu} \ar[dl]|y \ar[dr]|x 
&&\stackrel{\langle v_2 \rangle }{\bu} \ar[dl]|y \ar[dr]|x
&&\stackrel{\langle v_3 \rangle }{\bu} \ar[dl]|y\ar[dr]|x 
&&\stackrel{\langle v_4 \rangle }{\bu} \ar[dl]|y \ar[dr]|x\\
\bu \ar[dr]|x && \bu \ar[dl]|y \ar[dr]|x && \bu 
\ar[dl]|y \ar[dr]|x && \bu \ar[dl]|y \ar[dr]|x && \bu \ar[dl]|y \\
& \bu && \bu && \bu && \bu}}
\end{xy}
\end{equation}
The vertices correspond to  $k$-linear space generators, 
and the arrows indicate the action of the
generators $x$ and $y$ of the group algebra.  
A vertex with no out-coming arrows  corresponds to a 
trivial 1-dimensional submodule.

	If $p = 5$, then $W ~\simeq ~ I^3/I^6$, and
$W$ has constant Jordan type $3[3] + 2[2]$.   More generally, the Jordan type
of the $kE$-module $I^{p-2}/I^{p+1}$ is $(p-2)[3] + 2[2]$.  

If $p > 5$, then $W$ does not have constant Jordan type.   Namely, the
Jordan type for both $x$ and $y$ on $W$ is $3[3]+2[2]$, 
whereas the Jordan type of $x+y$ is $4[3]+1[1]$.
The Jordan blocks of size 3 for the action of $x+y$  are generated by 
$v_1, v_2, v_3$ and $v_4$ and the trivial block is  generated by
$x(v_1-v_2+v_3-v_4) + y(v_4)$.
\end{ex}

\begin{prop}
Let $sl_2$ denote the $p$-restricted Lie algebra of 
$2 \times 2$ matrices of trace 0
and let $u(sl_2)$ denote the restricted 
enveloping algebra of $sl_2$, the group 
algebra of the height 1 infinitesimal group scheme $G = SL_{2(1)}$.  
If the finite dimensional $u(sl_2)$-module $M$ 
is the restriction of a rational 
$SL_2$-module, then $M$ has constant Jordan type.  On the other hand, 
other $u(sl_2)$-modules typically do not have constant Jordan type.
\end{prop}

\begin{proof}  Let $\xymatrix@-.5pc{\rho: SL_2 \ar[r] & GL_m}$ 
determine a rational $SL_2$-module
$M$ of dimension $m$.  Then the induced map of $p$-restricted Lie algebras
$\xymatrix@-.5pc{d\rho: sl_2 \ar[r] & gl_m}$ defines the associated 
$u(sl_2)$-module structure on $M$.  For any field extension $K/k$ and any
$ x \in SL_2(K)$, the rational $SL_{2,K}$-module $M_K^{\rho(x)}$ given by 
$Ad(\rho(x)) \circ \rho$ is isomorphic to $M_K$ and thus the associated
$u(sl_{2,K})$-module $M_K^{\rho(x)}$ given by $d(Ad(\rho(x) \circ \rho)$ is 
isomorphic to $M_K$.

Recall that the space of $k$-rational points of 
$\Pi(G)$ can be identified with 
the space of $k$-rational lines of the nilpotent variety $\N(sl_2)$.
Indeed, each equivalence
class of $\pi$-points  is represented by some 
$\xymatrix@-.5pc{\alpha_K: K[t]/t^p \ar[r] & u(sl_{2,K})}$ sending
$t$ to a nilpotent matrix of $sl_2(K)$.  Moreover,
the Jordan type of a given finite
dimensional $u(sl_2)$-module does not depend upon the choice of such 
a representative of a given point of $\Pi(G)$ by \cite[3.1]{FPS}. 
The action of $SL_2(K)$ on $\Pi(G)_K$ sending
$\alpha_K$ to  $d\rho(Ad(x))\circ \alpha_K = 
Ad(\rho(x)) \circ \alpha_K$ corresponds to the 
natural action of $SL_2$ on $\N(sl_2)$.  Consequently, the 
transitivity of this action 
together with Proposition \ref{rep} implies that the rational $SL_2$-module
$M$ has constant Jordan type as a $u(sl_2)$-module.

On the other hand, any $u(sl_2)$-module $M$ whose $\Pi$-support
space $\Pi(G)_M$ is a non-empty proper subset of the 1-dimensional
variety $\Pi(G)$ has
Jordan type of $M$ at a generic $\pi$-point of $\Pi(G)$ equal to  
$\frac{m}{p}[p]$ (where $m = \Dim M$) and strictly smaller Jordan type at any
$\pi$-point representing a point in $\Pi(G)_M$. 
Such $M$ abound, since every finite subset of $\Pi(G)$ 
is of the form $\Pi(G)_M$
for some finite dimensional $kG$-module $M$.   
For example, let $b \subset sl_2$ be the Lie subalgebra of lower 
triangular matrices and consider the $u(sl_2)$-module 
$M = u(sl_2) \otimes_{u(b)} k$,
the module obtained from the trivial $u(b)$-module by coinduction.  Then
$M$ is free when
restricted to the 1-dimensional subalgebra of 
strictly upper triangular matrices
and trivial when restricted to the 1-dimensional subalgebra of strictly lower
triangular matrices.
\end{proof}

We conclude this section of explicit examples 
with a family of modules $V_n, n > 0$
which are modules of constant Jordan type 
$n[2] + 1[1]$ regardless of the 
prime $p$.  The module $V_n$ can be 
represented by the following diagram:
$$
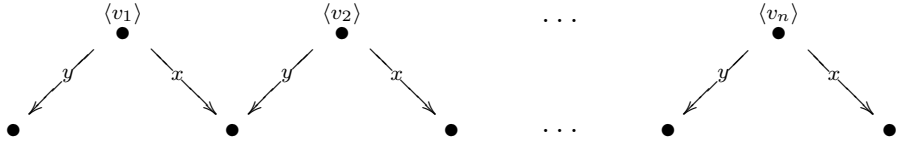
\begin{xy}*!C\xybox{%
\xymatrix{ &\stackrel{\langle v_1 \rangle }{\bu} 
\ar[dl]|y \ar[dr]|x &&
\stackrel{\langle v_2 \rangle }{\bu} \ar[dl]|y \ar[dr]|x
&&\dots &&\stackrel{\langle v_n \rangle }{\bu} \ar[dl]|y \ar[dr]|x\\
\bu && \bu && \bu &\dots &  \bu && \bu}}
\end{xy}
$$
The verification of the assertion 
that $V_n$ does indeed have constant
Jordan type is a simple computation.

In the following proposition we describe the modules $V_n$
explicitly in terms of generators and relations. 

\begin{prop}
 Consider a rank 2 elementary abelian $p$-group $E$,
so that $kE= k[x,y]/(x^p,y^p)$.  Consider the $kE$-module $V_n$ of 
dimension $2n+1$ generated by $v_1,\ldots,v_n$ spanned as a 
$k$-vector space by
$\{ v_1,\ldots,v_n,x(v_1),\ldots,x(v_n),y(v_1)\}$ with 
$x(v_i) = y(v_{i+1}), ~ x^2(v_i) = xy(v_i) = y^2(v_i) = 0$. Then $V_n$ 
has constant Jordan type $n[2] + 1[1]$.
\end{prop}


\section{Specialization, $\Gamma(G)_M$, and constant Jordan
type}
\label{sp}

In this section, we introduce a strict form of specialization 
of $\pi$-points which was
considered briefly in \cite[3.2]{FPS} for infinitesimal group schemes.
We also recall the non-maximal support variety $\Gamma(G)_M$ of a finite
dimensional $kG$-module $M$.  We use both strict specialization and
the non-maximal support variety to further explore the class of modules
of constant Jordan type.

The condition recalled in Definition \ref{pi}
that the $\pi$-point $\alpha_K$ specializes to the $\pi$-point
$\beta_L$ is equivalent to the algebro-geometric condition
that the point $[\alpha_K] \in \Pi(G)$ 
specializes to $[\beta_L]$ (\cite{FP2}).
For some purposes, this notion of specialization is not sufficiently strong. 
Namely, if $\Pi(G)$ is reducible, then 
there exist finite dimensional $kG$-modules $M$ and 
$\pi$-points $\alpha_K, \beta_L$
such that $\alpha_K$ specializes to $\beta_L$ but
the Jordan type of $\alpha_K^*(M_K)$ is smaller than 
the Jordan type of $\beta_L^*(M_L)$ (see \cite[4.14]{FP2}).

The following definition introduces a stricter definition of specialization
which is a natural extension 
of \cite[3.2]{FPS} from infinitesimal group schemes to arbitrary finite 
group schemes. 

\begin{defn}
\label{strict}
Let $G$ be a finite group scheme and $\alpha_K, \beta_L$ be 
$\pi$-points of $G$.  We say that $\alpha_K$ strictly specializes to
$\beta_L$ (and write $\alpha_K \downdownarrows \beta_L$),
if there exists a commutative local domain $R$ over $k$ with field
of fractions $K$ and residue field $L$, together with a map of
$R$-algebras $\xymatrix@-.5pc{\nu_R: R[t]/t^p \ar[r] & RG}$ 
such that $\nu_R \otimes_R K
= \alpha_K, ~ \nu_R \otimes_R L = \beta_L$.
\end{defn}

\begin{thm}
\label{matrix}
Let $R$ be a commutative local domain with field of fractions $K$ and 
residue field $L$.
Let $A \in M_N(R)$ be an $N\times N$ matrix 
which has $p$-th power $0$ and 
coefficients in  $R$.  Then the 
Jordan type of $A\otimes_R K \in M_N(K)$ is greater or equal to 
the Jordan type of
$A\otimes_RL \in M_N(L)$.

Consequently, if $\alpha_K, \beta_L$ are $\pi$-points of a finite group 
scheme $G$ with $\alpha_K \downdownarrows \beta_L$ and if $M$ is
a finite dimensional $kG$-module, then $\alpha_K^*(M_K)$ has Jordan
type greater or equal to that of $\beta_L^*(M_L)$.  In particular, 
$\alpha_K \downdownarrows \beta_L$ implies
$\alpha_K \downarrow \beta_L$.
\end{thm}

\begin{proof}
Let $M[i]$ denote the cokernel of 
$\xymatrix@-.5pc{A^i: R^N \ar[r] & R^N}$ for some $i < p$.  
Let $m_1,\ldots,m_s \in M[i]$ be chosen so that 
$\ol m_1,\ldots, \ol m_s \in M[i]\otimes_R L$ is a basis,
where $\ol m_i = m_i \otimes 1_L$.  By Nakayama's
Lemma,  $m_1,\ldots,m_s$ generate $ M[i]$ as an $R$-module and thus
their images  span $M[i] \otimes_R K$.  Observe that $M[i]\otimes_R K$ 
is the cokernel of $\xymatrix@-.5pc{(A\otimes_R K)^i: K^N \ar[r] & K^N}$ 
and that $M[i]\otimes_R L$ 
is the cokernel of the homomorphism 
$\xymatrix@-.5pc{(A\otimes_R L)^i: L^N \ar[r] & L^N}$.  
This implies that the rank of 
$(A\otimes_R K)^i$ is greater or equal to the rank of $(A\otimes_R L)^i$ 
for any $i < p$ so that  the 
Jordan type of $A\otimes_R K \in M_N(K)$ is greater or equal to 
the Jordan type of
$A\otimes_RL \in M_N(L)$.

The second statement is a special case of the first.
The last statement follows from the observation that if $\beta_L^*(M_L)$ is 
projective then its Jordan type is the  maximal
possible Jordan type on $M$ and hence  $\alpha_K^*(M_K)$
must also be projective assuming that $\alpha_K \downdownarrows \beta_L$.
\end{proof}

The next theorem verifies that specialization of points in $\Pi(G)$ can
be represented by strict specialization.

\begin{thm}
\label{spec}
Let $G$ be a finite group scheme and let $\alpha_K, ~ \beta_L$ be
$\pi$-points of $G$ with $\alpha_K \downarrow \beta_L$.  Then there exist 
$\pi$-points $\alpha_{K^\prime}^\prime
\sim \alpha_K$ and $\beta_{L^\prime}^\prime \sim \beta_L$ such that
$\alpha_{K^\prime}^\prime \downdownarrows \beta_{L^\prime}^\prime .$
\end{thm}

\begin{proof}
Using \cite[4.13]{FP2}, we can choose 
an elementary abelian $p$-group $E \subset \pi_0(G)$ 
such that $[\alpha_K], [\beta_L]$ are in the image of the closed
map $\xymatrix@-.5pc{\Pi((G^0)^E \times E) \ar[r] & \Pi(G)}$.  Observe that 
the group algebra of $(G^0)^E \times E$ is isomorphic to the group algebra
of some infinitesimal group scheme $H$.   Because the relationship of
strict specialization given in Definition \ref{strict} and the definition of 
specialization given in Definition \ref{pi}
involve only the group algebra of the given finite group scheme, we may
replace $G$ by $H$.  Thus, we assume $G$ is an infinitesimal group
scheme.

Let $\Spec A \subset \Pi(G)$ be an irreducible affine open subset containing
$[\beta_L]$.  Since $\alpha_K \downarrow \beta_L$, we have that $[\beta_L]$ 
is in the closure 
of $[\alpha_K]$ by \cite[4.3]{FP2}.  Hence, any closed irreducible component 
in the complement of $\Spec A$ containing $[\alpha_K]$ must contain 
$[\beta_L]$.  We conclude that $[\alpha_K] \in \Spec A$.

Replacing $A$ by $A_{\rm red}$, we may
assume that $A$ is a domain.  Let $R$ denote the local $A$-algebra defined
as the localization at the prime corresponding to $[\alpha_K]$ of the 
quotient of $A$ by the prime corresponding to $[\beta_L]$.  Set
$K^\prime$ to be the field of fractions of $R$ and $L^\prime$ to be
the residue field of $R$.  

Let $r$ denote the height of the infinitesimal group scheme $G$ and let
$V_r(G)$ be the scheme of 1-parameter subgroups of $G$. Recall that the
natural morphism $\xymatrix@-.5pc{\Theta_G:  V_r(G)\backslash \{ 0 \} 
\ar[r] & \Pi(G)}$, which is
determined by sending a 1-parameter subgroup 
$\xymatrix@-.5pc{\mu: \bG_{a(r),K}  \ar[r] & G_K}$
to the $\pi$-point 
$\xymatrix@-.6pc{\mu_* \circ \epsilon: K[t]/t^p \ar[r] &
K\bG_{a(r)} \ar[r] & KG}$, where $\epsilon$ has the property
that its composition with the map on group algebras induced
by the projection $\xymatrix@-.5pc{\bG_{a(r)} 
\ar[r]& \bG_{a(r)}/\bG_{a(r-1)}}$, is an isomorphism. 
Because  $\Theta_G$ together with the isogeny $\HHH^\bu(G,k) \to k[V_r(G)]$ 
of \cite[5.2]{SFB} induces the isomorphism 
$\xymatrix@-.5pc{\Proj(V_r(G)) \ar[r] & \Pi(G)}$
of Theorem \ref{iso}, we conclude that the given morphism 
$\xymatrix@-.5pc{\Spec R \ar[r] & \Pi(G)}$ lifts to a 
morphism 
$$
\xymatrix{\Spec R \ar[r] & (\Spec H^\bu(G,k)[1/P])_0 \subset V_r(G),}
$$ 
which corresponds to a morphism of $R$-group schemes 
$\xymatrix@-.5pc{\nu: \bG_{a(r),R} \ar[r] & G_R}$. Here, $P$ is a 
homogeneous element of positive degree of $\HHH^\bu(G,k)$ which does not 
vanish on the image of the closed point of $\Spec R$ and $(A^\bu)_0$ 
denotes the subalgebra of degree 0 elements of the graded algebra $A^\bu$.

We define $\nu_R$ to be the map of $R$-algebras given as the composition 
$$
\xymatrix{
R[t]/t^p  \ar[r]^{\epsilon} & R(\bG_{a(r),R}) \ar[r]^{\quad\quad\nu_*} & RG.
}
$$ 
By construction, $\nu_R \otimes_R K^\prime$ is a $\pi$-point of $G$ in the  
equivalence class $[\alpha_K]$ (the image of $\Spec K^\prime$) and 
$\nu_R \otimes_R L^\prime$ is a $\pi$-point of $G$ in the 
equivalence class $[\beta_L]$.
\end{proof}

The proof of the theorem in the title of \cite{C1} 
applies with only minor notational change to prove
the same assertion, given below, for an arbitrary finite group scheme.

\begin{thm} (\cite{C1})
\label{connected}
Let $G$ be a finite group scheme and $M$ a finite dimensional 
indecomposable $kG$-module.  
Then $\Pi(G)_M$ is connected.
In particular, $\Pi(G) = \Pi(G)_k$ is connected.
\end{thm}

Theorems \ref{spec} and \ref{connected} enable the following
equivalent formulation of the condition that a module has constant Jordan type.

\begin{prop}
\label{explicit}
Let $G$ be a finite group scheme and let $M$ be a $kG$-module.
Then $M$ has constant Jordan type if and only if 
$\alpha_K$ and  $\beta_L$ have the same Jordan type on $M$ 
whenever $\alpha_K, ~ \beta_L$
are $\pi$-points of $G$ satisfying
$\alpha_K \downdownarrows \beta_L$.
\end{prop}

\begin{proof}
If $M$ has constant Jordan type and  $\alpha_K, ~ \beta_L$
are $\pi$-points of $G$, then the Jordan types of 
$\alpha_K$ and  
and $\beta_L$ on $M$ are necessarily equal. 

To prove the converse, we recall that if $x \in \Pi(G)$ is a point such
that $M$ has maximal Jordan type at some representative of $x$, then 
the Jordan type on $M$ is the same for every representative 
of $x$ (see Remark~\ref{indep}).  
Let $x \in \Pi(G)$ be a point which has maximal Jordan type on $M$ and let
$y \in \Pi(G)$ be an arbitrary point.   Since $\Pi(G)$ is connected (by 
Theorem \ref{connected}) and Noetherian, we may find a chain of 
irreducible components $Z_0, \ldots, Z_r$ of $\Pi(G)$ such that 
$Z_i \cap Z_{i+1} \not= \emptyset$ and $x \in Z_0, y \in Z_r$.  Let
$\eta_i \in Z_i$ be the generic point and choose 
some $z_i \in  Z_i \cap Z_{i+1}$.
Then $x,\eta_0,z_1,\eta_1,\ldots, z_i,\eta_i,\dots,\eta_r,y$
is  a chain  of points $x=x_0,x_1, \ldots,x_n=y \in \Pi(G)$ such that there 
exist morphisms
$\xymatrix@-.5pc{\Spec R_i\ar[r] & \Pi(G)_{\rm red}}$ sending 
$\{ \Spec K_i, \Spec L_i \}$ 
to the (unordered)
pair $\{ x_{i-1},x_i \}$ where $R_i$ is a commutative 
local domain with field 
of fractions $K_i$
and residue field $L_i$.    

Our hypothesis in conjunction with Theorem \ref{spec}
implies that there are representing $\pi$-points $\alpha_{K_i}, ~ \alpha_{L_i}$
of $x_i$ and  $x_{i-1}$ at which $M$ has the same Jordan type. 
Consequently, $M$ has the same 
maximal Jordan type at each representative of each $x_i$.
Thus, the Jordan type of $M$ at any representative of $x_n = y$ equals this
same maximal Jordan type; in other words, $M$ has constant Jordan type. 
\end{proof}

We give a useful characterization of modules $M$  of constant Jordan type
in terms of their non-maximal support varieties $\Gamma(G)_M$ recalled in
Remark \ref{indep}.

\begin{prop}
\label{empty}
Let $M$ be a finite dimensional $kG$-module.  Then $M$ has 
constant Jordan type if and only if $\Gamma(G)_M = \emptyset$.
\end{prop}

\begin{proof}
If $M$ has constant Jordan type, then clearly $\Gamma(G)_M = \emptyset$.
Conversely, let $\Gamma(G)_M = \emptyset$,  and let 
$\alpha_K$, $\beta_L$ be any two $\pi$-points of $G$ 
satisfying   $\alpha_K \downdownarrows \beta_L$.  By 
Theorem~\ref{matrix}, $\alpha_K^*(M_K)$ has Jordan
type greater or equal to that of $\beta_L^*(M_L)$. 
Since $\Gamma(G)_M = \emptyset$, we must have that 
the Jordan types of  $\alpha_K^*(M_K)$ and $\beta_L^*(M_L)$ 
are the same.    Hence, 
Proposition~\ref{explicit} implies that $M$ has constant Jordan type. 
\end{proof}

The following ``closure property" of modules of constant Jordan type
is somewhat striking.

\begin{thm}
\label{summand}
Let $M$ be a $kG$-module of constant Jordan type.  Then any 
direct summand of $M$
also has constant Jordan type.
\end{thm}

\begin{proof}
Write $M = M^\prime \oplus M^{\prime \prime}$.  
Let $\alpha_K$, $\beta_L$ be two $\pi$-points such that 
$\alpha_K \downdownarrows \beta_L$.   
By Theorem \ref{matrix},  we have 
\begin{equation}
\label{mprime}
\JType(\alpha_K^*(M^\prime_K)) \geq \JType(\beta_L^*(M^\prime_L))
\text{ and } 
\JType(\alpha_K^*(M^{\prime\prime}_K)) \geq 
\JType(\beta_L^*(M^{\prime\prime}_L)).
\end{equation}
Hence, 
$$
\JType(\alpha_K^*(M^\prime_K)) \oplus \JType(\alpha_K^*(M^{\prime\prime}_K))
= \JType(\alpha_K^*(M)) = \JType(\beta_L^*(M_L)) =  
$$
$$ 
= \JType(\beta_L^*(M^\prime_L)) \oplus 
\JType(\beta_L^*(M^{\prime\prime}_L)),
$$
where $\JType(-)$ is the Jordan type function.
Therefore, both inequalities in (\ref{mprime}) must be equalities. 
Since this holds for any pair $\alpha_K \downdownarrows \beta_L$, 
the statement now follows immediately from Proposition \ref{explicit}.
\end{proof}

\vskip .2in


\section{Behavior with respect to tensor products}
\label{behave}

We begin with the following ``order preserving'' property of Jordan 
types.  The proof given below uses an explicit description of the
tensor product of indecomposable $k[t]/t^p$-modules which is presented
in the Appendix.  The coalgebra structure we (implicitly) give $k[t]/t^p$
is that for which $t$ is primitive (i.e., 
$\nabla(t) = t \otimes 1 + 1 \otimes t$).

\begin{prop}
\label{cyclic}
Let $M$, $N$ and $L$ be $k[t]/t^p$-modules with Jordan types
$\ul a = a_p[p] + \cdots + a_1[1], ~ \ul b = b_p[p] + \cdots + b_1[1]$
and $\ul c = c_p[p] + \cdots + c_1[1]$ respectively such that 
$\Dim M ~ = ~ \Dim N$.  Let 
$\ul{a} \otimes \ul{c},~ \ul {b} \otimes \ul{c}$ denote the Jordan types of
$M\otimes L, ~ N\otimes L$ respectively.

If $\ul a ~ \geq ~ \ul b$, then $\ul a\otimes \ul c ~ 
\geq ~ \ul b \otimes \ul c$.

If $\ul a ~ > ~ \ul b$ and if $c_i \not=0$ for some $i < p$, 
then $ \ul a \otimes \ul c ~ > ~ \ul b \otimes \ul c$.

\end{prop}

\begin{proof} 
For a $k[t]/t^p$-module $M$ of dimension $m$, we denote 
the representation afforded by $M$ by 
$\xymatrix@-.5pc{\rho_M:k[t]/t^p \ar[r] & \End_k(A) \simeq M_m(k)}$.
The dominance condition (\ref{dom}) on Jordan types $\ul a ~ \geq ~ \ul b$ 
of $M$ and $N$ both of dimension $m$ can be formulated as the 
condition that
$$\rk \rho_M(t)^i \geq \rk \rho_N(t)^i, \quad 1 \leq i < p$$
or equivalently that 
\begin{equation}
\label{kerr}
\Dim \Ker \{\rho_M(t)^i\} \leq \Dim \Ker \{ \rho_N(t)^i \},
\quad 1 \leq i < p.
\end{equation}
For $\ul a ~ > ~ \ul b$, the additional condition is that 
$\rk \rho_M(t)^i > \rk \rho_N(t)^i$ for some $i$.

If $L \simeq L_1 \oplus L_2$, so that 
$\ul c ~= ~ \ul c_1 + \ul c_2$, 
then we immediately conclude that 
$$\ul a \otimes \ul c ~ = ~
\ul a \otimes \ul c_1 ~ + ~ \ul a \otimes \ul c_2.$$  
Thus,  we may assume that $L$ is indecomposable, say $L = [\ell]$.

Denote  by $J_s \in M_s(k)$ the Jordan block of size $s$.  Then 
$J_s = \rho_{[s]}(t)$ once we choose an appropriate basis. 
Since $t$ acts on $M\otimes [\ell]$ as $ t \otimes 1 + 1 \otimes t$, 
we get the formula
$$\rho_{M\otimes [\ell]}(t) = \rho_M(t)\otimes \Id_{[s]} + \Id_M \otimes J_s$$
Lemma 1.10 of \cite{FPS} implies 
\begin{equation}\label{ker}
\Ker \{\rho_{M\otimes [\ell]}(t)\} ~ \cong ~ \Ker \{\rho_M(t)^\ell\}
\end{equation}
Applying the same argument to $M\otimes [\ell]$, we obtain
\begin{equation}
\label{ker2}
\Ker \{\rho_{M\otimes [\ell]}(t)^s \}= 
\Ker \{\rho_{M\otimes[\ell]\otimes[s]}(t)\}
\end{equation}
We write $[\ell]\otimes[s] = \bigoplus_j C_{\ell s}^j [j]$ 
where the coefficients $C_{\ell s}^j$ are given by the 
Corollary \ref{tensor2} of the Appendix.  
With this notation,
$$\Dim \Ker \{ \rho_{M\otimes [\ell]}(t)^s \} = \Dim \Ker 
\{\rho_{M\otimes (\bigoplus C_{\ell s}^j [j])}(t)\} = 
\Dim \Ker \{\rho_{\bigoplus C_{\ell s}^j (M \otimes [j])}(t)\}$$
$$ = 
\sum C_{\ell s}^j \Dim \Ker \{\rho_{M \otimes [j]}(t)\} = 
\sum C_{\ell s}^j \Dim \Ker \{\rho_M(t)^j\}.
$$
Consequently,
 $$ \Dim \Ker \{\rho_{M\otimes [\ell]}(t)^s\} = 
\sum C_{\ell s}^j \Dim \Ker \{\rho_M(t)^j\}$$
 and
$$ \Dim \Ker \{\rho_{N\otimes [\ell]}(t)^s\} = 
\sum C_{\ell s}^j \Dim \Ker \{\rho_N(t)^j\}
$$
Hence, (\ref{kerr}) implies that  
 $$\Dim \Ker \{\rho_{M\otimes [\ell]}(t)^s \}\ \leq \Dim \Ker 
\{\rho_{N\otimes [\ell]}(t)^s\}$$
 for all $s$.  Therefore, 
$\ul a \otimes \ul c ~ \geq ~ \ul b \otimes \ul c$.

Now assume $\ul a ~ > ~ \ul b$ and that some $c_i \not= 0, i < p$.
Then, we may assume that $L = [\ell]$ with $\ell < p$.
Choose $j$ such that 
$\Dim \Ker \{\rho_M(t)^j\} < \Dim \Ker \{\rho_N(t)^j\}$.  
The formula of Corollary \ref{tensor}  implies 
that there exists some $s$ such that $C_{\ell s}^j = 1$. 
That is, if $j \geq \ell$, then
take $s = j-\ell+1$, and if $j < \ell$, take $s = \ell-j +1$.   
For such $s$ we get the strict inequality 
 $$\Dim \Ker \{\rho_{M\otimes [\ell]}(t)^s \} ~ < ~  \Dim 
\Ker \{\rho_{N\otimes [\ell]}(t)^s\}$$
 Therefore, $\ul a \otimes \ul c ~ > ~ \ul b \otimes \ul c$.

\end{proof}

Proposition \ref{cyclic} enables us to establish various tensor product 
properties of maximal Jordan type.  We should point out that there are
some subtleties which must be carefully considered. For example if 
$M$ and $N$ are $kG$-modules and $\alpha_K$ is a $\pi$-point whose 
image is not a sub-Hopf algebra of $KG$, then it is not always true 
that $\alpha_K^*(M_K \otimes N_K) ~ \simeq ~ \alpha_K^*(M_K) \otimes
\alpha_K^*(N_K).$  Hence, Proposition \ref{cyclic} can not be 
applied directly. The somewhat  surprising Examples \ref{tensor1}
and \ref{tensor0} indicate the subtle relationship between tensor products 
and maximal Jordan types.
In view of these examples, we take some care in the proofs of 
the following properties.
\begin{thm}
\label{max}
Let $G$ be a finite group scheme, 
and consider two finite
dimensional $kG$-modules $M, ~ N$ and a $\pi$-point $\alpha_K$ of $G$.
\begin{enumerate}
\item
If $\alpha_K$ has  maximal Jordan type on both $M$ and $N$,  
then
$$\alpha_K^*(M_K \otimes N_K) ~ \simeq ~ \alpha_K^*(M_K) \otimes 
\alpha_K^*(N_K).$$
\item 
If  $\alpha_K$ 
has maximal Jordan type on $M \otimes N$, then
$$
\alpha_K^*(M_K \otimes N_K) ~ \simeq ~ \alpha_K^*(M_K) \otimes 
\alpha_K^*(N_K),
$$
\item
If $\Pi(G)$ is irreducible and if $\alpha_K$ 
has maximal Jordan type on both $M$ and $N$, then 
$\alpha_K$ has maximal Jordan type on $M \otimes N$ and that Jordan
type is equal to the Jordan type of
$\alpha_K^*(M_K) \otimes \alpha_K^*(N_K)$.
\end{enumerate}
\end{thm}

\begin{proof}
Let $i: C_K \subset G_K$ be a unipotent abelian subgroup scheme through which
$\alpha_K$ factors.  Since $\xymatrix@-.5pc{i: KC_K \ar[r] & KG_K}$ 
is a map of Hopf algebras,
$i^*$ commutes with tensor products.  Observe that
the maximality of $\alpha_K$ on  $M, N$ or $M\otimes N$ as 
$kG$-modules implies the maximality of $\alpha_K$ on  $M_K, N_K$ or
$M_K\otimes_K N_K$ as $KC_K$-modules.  Because $\Pi(C_K)$ is
irreducible, the maximal Jordan types of $M_K, N_K$, and 
$M_K\otimes_K N_K$ are all achieved at a generic $\pi$-point of 
$\Pi(C_K)$.
Thus, statement (1) follows from  \cite[4.4]{FPS} applied to $C_K$.

To prove (2), we extend scalars if necessary so that $K$ is perfect. 
Hence,  $KC_K \simeq K[T_1,\ldots,T_r]/(T_1^{p^{e_1}},
\ldots,T_r^{p^{e_r}})$ (see \cite[14.4]{W}).
Let $ t_i ~ =  ~ T_i^{p^{e_i-1}}$.  

Let $\xymatrix{\rho_M: KC_K \ar[r]& \End_K(M_K)}$ 
be the map defined by the representation $M_K$ 
of $C_K$, and similarly define $\rho_N$ and $\rho_{M\otimes N}$.
Choose  a generic $\pi$-point $\eta_\Omega$ of $C_K$,
$\xymatrix@-.6pc{\eta_\Omega: \Omega[t]/t^p \ar[r] & \Omega C_K}$, 
where $\Omega/K$ is a field extension. 
Arguing exactly as in the proof of \cite[4.4]{FPS}, we  conclude 
that the Jordan type of 
$\rho_M(\eta_\Omega(t))\otimes 1 + 1 \otimes \rho_N(\eta_\Omega(t))$ 
as an element of $\End_\Omega(M_\Omega\otimes N_\Omega)$ 
is greater or equal than the Jordan type  of any linear 
combination of elements of the form 
$\rho_M(t_i) \otimes 1$ and $1 \otimes \rho_N(t_j)$, $1 \leq i,j \leq r$.
Since $\eta^*_\Omega(M_\Omega \otimes N_\Omega) 
\simeq \eta^*_\Omega(M_\Omega) \otimes \eta^*_\Omega(N_\Omega)$ 
by \cite[4.4]{FPS},  
the same conclusion holds for $\rho_{M\otimes N}(\eta_\Omega(t))$. 
Since $\alpha_K$ is maximal on 
$M\otimes N$ and $\eta_\Omega$ is a generic $\pi$-point,  
the Jordan type of $\alpha^*_K(M_K\otimes N_K)$ 
equals the Jordan type  of $\eta^*_\Omega(M_\Omega\otimes N_\Omega)$. 
Hence, the Jordan type of   
$\rho_{M\otimes N}(\alpha_K(t))$ is greater or equal than the 
Jordan type of any linear combination of elements of the form 
$\rho_M(t_i) \otimes 1$ or $1 \otimes \rho_N(t_j)$, $1 \leq i,j \leq r$. 
\sloppy
{

}

Let $\nabla: KC_K \to KC_K \otimes KC_K$ be the coproduct on $KC_K$, 
and let $I=(T_1, \ldots, T_r)$ be the augmentation ideal of $KC_K$. 
and let $I^{(p)}$ 
be the ideal generated by $(t_1, \dots, t_r)$.  Recall  that 
$$\alpha_K(t) \otimes 1 + 1 \otimes \alpha_K(t) - 
\nabla(\alpha_K(t)) \in I \otimes I
$$
(see \cite[I.2.4]{J}). 
Since $\alpha_K(t)$ has $p$-th power 0, 
we can refine this further, concluding that  
\begin{equation}
\label{nabla}
\alpha_K(t) \otimes 1 + 1 \otimes \alpha_K(t) - \nabla(\alpha_K(t))
\in  I \otimes I^{(p)} + I^{(p)} \otimes I.
\end{equation} 
  The action of $KC_K$ on 
the tensor product $M_K \otimes N_K$ is given by the formula
$$\rho_{M\otimes N} = (\rho_M \otimes \rho_N)  \circ \nabla$$
Hence, 
\begin{equation}
\label{rel} 
\rho_M(\alpha_K(t))\otimes 1 +  1 \otimes  \rho_N(\alpha_K(t)) =
\end{equation}
$$\rho_{M \otimes N}(\alpha_K(t)) +  \langle\text{multiples of } 
\rho_M(t_i) \otimes 1  \text{ and } 1 \otimes \rho_N(t_j)\rangle
$$

Since the Jordan type of $\rho_{M \otimes N}(\alpha_K(t))$ 
is greater of equal than the Jordan type of any linear 
combination of elements 
$\rho_M(t_i) \otimes 1$   and 
$1 \otimes \rho_N(t_j)$ for $1 \leq i,j \leq r$, 
the relation (\ref{rel}) and Theorem 1.12 of \cite{FPS} 
imply that $\rho_{M \otimes N}(\alpha_K(t))$ and 
$\rho_M(\alpha_K(t))\otimes 1 +  1 \otimes  \rho_N(\alpha_K(t))$ 
have the same Jordan type.   Hence,  
$\alpha^*_K(M) \otimes_K \alpha^*_K(N) \simeq \alpha^*_K(M_K \otimes_K N_K)$.

To prove (3) we assume that $\Pi(G)$ is irreducible. Let 
$\xymatrix@-.5pc{\eta_\Omega: \Omega[t]/t^p \ar[r] & \Omega G}$ 
be a generic $\pi$-point of 
$G$.  Applying \cite[4.4]{FPS} to $\eta_\Omega$, we get 
$$\eta_\Omega^*(M_\Omega) 
\otimes _\Omega \eta_\Omega^*( N_\Omega) ~ \simeq ~
\eta_\Omega^*(M_\Omega \otimes _\Omega N_\Omega) $$
Since $\Pi(G)$ is irreducible, the absolute 
maximal Jordan type of any finite dimensional $kG$-module
is realized at $\eta_\Omega$.  Hence, the 
maximality assumption on $\alpha_K$ on $M$ and $N$  implies that 
$\alpha_K^*(M_K)$ has the same Jordan type as $\eta_\Omega^*(M_\Omega)$ 
and $\alpha_K^*(N_K)$ has the same Jordan type as $\eta_\Omega^*(N_\Omega)$.
We conclude that 
$$
\alpha_K^*(M_K \otimes N_K) ~ \simeq ~ \alpha_K^*(M_K) \otimes 
\alpha_K^*(N_K)
$$
has the same Jordan type as 
$$
\eta_\Omega^*(M_\Omega) 
\otimes _\Omega \eta_\Omega^*( N_\Omega) ~ \simeq ~
\eta_\Omega^*(M_\Omega \otimes _\Omega N_\Omega) 
$$
where the first isomorphism follows from statement (1). 
Hence, the Jordan type of $\alpha_K$ on $M \otimes N$ is maximal. 
\end{proof}

An easy corollary of  Theorem \ref{max} is the following assertion
that the tensor 
product of modules of constant Jordan type is again of constant Jordan type.

\begin{cor}
\label{tensor}
Let $G$ be a finite group scheme and let $M, ~N$ be finite dimensional
$kG$-modules.  If $M$ and $ N$ have constant Jordan type, then $M \otimes N$
also has constant Jordan type.
\end{cor}

\begin{proof}
We merely observe that $\alpha_K$ has maximal
Jordan type on both $M$ and $N$ for any $\pi$-point 
$\xymatrix@-.5pc{\alpha_K: K[t]/t^p \ar[r] & KG}$ whenever 
$M, ~N$ are of constant Jordan type.  Thus, the corollary follows from the
first statement of Theorem \ref{max}.
\end{proof}

The following consequence of Theorem \ref{max} will be used to 
prove Proposition \ref{Gamma} .

\begin{cor}
\label{easy}
Let $G$ be a finite group scheme such that 
$\Pi(G)$ is irreducible and let 
$\xymatrix@-.5pc{\alpha_K: 
K[t]/t^p \ar[r] & KG}$ be a $\pi$-point of $G$.  
Let $M$ be a $kG$-module such that
$\alpha_K^*(M_K)$ is not projective and let $N$ be another $kG$-module such 
that $\alpha_K$ does not have maximal Jordan type on $N$.  Then 
$\alpha_K$ does not have maximal Jordan type on $M\otimes N$.
\end{cor}

\begin{proof}
If the Jordan type of $\alpha_K^*(M_K \otimes_K N_K)$ were maximal, 
then the second statement of
Theorem~\ref{max} would imply that this maximal type is the same as that of 
$\alpha_K^*(M_K) \otimes_K \alpha_K^*(N_K)$.  However, the hypotheses
on $\alpha_K^*(M_K), ~ \alpha_K^*(N_K)$ together with Proposition
\ref{cyclic} would then lead to an immediate contradiction.
\end{proof}

We give two examples to illustrate that naturally formulated improvements of
Theorem \ref{max} are not valid.  
Our first example illustrates that the maximality of both $M$ and 
$N$ at a given
$\pi$-point is not sufficient to imply the maximality of $M\otimes N$ at
that $\pi$-point.  Namely, we construct a module $W$ and a 
$\pi$-point $\beta_L$ such that 
$\beta_L$ has maximal Jordan type on $W$ but not on $W \otimes W$  
(even though we have 
$\beta_L^*(W_L \otimes W_L) \simeq \beta_L^*(W_L) \otimes_L 
\beta_L^*(W_L)$ 
by Theorem \ref{max}).   
As usual, this anomaly comes from the fact that the ordering on Jordan 
types is not total.

\begin{ex}
\label{tensor1}
Let $G$ be a finite $p$-group which has two conjugacy classes of 
maximal elementary abelian 
subgroups, represented by $E$ and $E^\prime$ respectively. 
Furthermore,  we require $E$ to be normal.   
Let $e = \vert E \vert$, $f = \frac{\vert G\vert}{\vert E\vert}$.
Assume that $p>3$.

For example, take $G$ to be the $p$-Sylow subgroup of the 
wreath product $\Z/p\Z \, \wr \, S_p$, so that
$G$ is isomorphic to $(\Z/p\Z)^p \rtimes \Z/p\Z$.  Then $G$ has two 
non-conjugate maximal elementary abelian $p$-subgroups: 
$E = (\Z/p\Z)^{p}$ which is normal and 
$$
F = (\Z/p\Z \times \Z/p\Z \times \dots \times 
\Z/p\Z)^{\Z/p\Z} \times \Z/p\Z \cong (\Z/p\Z)^{\times 2}.
$$

By Quillen stratification, $\Pi(G) = X \cup Y$ 
where $X = \Pi(E)/G$, ~ $Y = \Pi(F)/N_G(F)$. 
Let $[\alpha_K] \in X, ~ [\beta_L] \in Y$ be generic points.

Choose a homogeneous cohomology class $\xi \in \HHH^\bu(G,k)$ with the
property that $\xi$ vanishes on $[\beta_L]$ but does not vanish on
$[\alpha_K]$, and let $L_\xi$ be Carlson module which has the property
that the support of $L_\xi$ is the zero locus of $\xi$.  Set
$M \ = ~ \Ind_E^G (\Omega_E^1(k))$,  set $N ~ = \ L_\xi^{\oplus n}$
for some positive integer $n$ which is to be determined, and
set $W \ = \ M \oplus N$.
It was shown in Example \cite[4.13]{FPS}, that 
if we pick $n$ to satisfy the inequality
\begin{equation}
\label{last}
\frac{f}{p} < n < (p-1)\frac{f}{p}
\end{equation}
then  the Jordan types $\alpha^*_K(W_K)$ and $\beta_L^*(W_L)$ are 
maximal, 
incomparable generic Jordan types of $W$.

Let $d = \Dim W$.
We proceed to deduce a condition on $n$ which would ensure the inequality 
$$
\JType(\alpha_K^*(W_K^{ \otimes 2})) > \JType(\beta_L^*(W_L^{ \otimes 2}))
$$ 
By the Appendix, $[p-1]\otimes [p-1] = (p-2)[p] + 1[1]$. 
Since $\alpha_K^*(W_K) = m[p] + f[p-1] $ for some $m$ (see  
\cite[4.13]{FPS}), and has dimension $d$, we get 
$$
\alpha_K^*(W_K^{ \otimes 2}) \simeq (\alpha_K^*(W_K))^{ \otimes 2} \simeq 
(\frac{d^2 - f^2}{p})[p] + f^2[1].
$$  
Similarly, applying the decomposition of  
$\beta_L^*(W_L)$ obtained in \cite[4.13]{FPS},
we get 
$$
\beta_L^*(W_L^{ \otimes 2}) \simeq (\beta_L^*(W_L))^{ \otimes 2} \simeq 
(\frac{d^2 - 2n^2p}{p})[p] + 2n^2[p-1] + 2n^2[1].
$$
In order for the Jordan type of $\alpha_K^*(W_K^{ \otimes 2})$  to dominate 
that of $\beta_L^*(W_L^{ \otimes 2})$ it suffices to choose $n$ such that 
$\alpha_K^*(W_K^{ \otimes 2})$ has 
more blocks of size  $p$ than $\beta_L^*(W_L^{ \otimes 2})$ and 
fewer blocks altogether. In other words, we 
need for the following inequalities to hold: 
$$
\frac{d^2 - f^2}{p} > \frac{d^2 - 2n^2p}{p}
$$
comparing  the number of blocks of size $p$,
and
$$\frac{d^2 - f^2}{p} + f^2 < \frac{d^2 - 2n^2p}{p} + 4n^2$$
comparing  the overall number of blocks.  
Simplifying, we get
$$
f^2 < 2n^2p,
$$
$$
(p-1) f^2 < 2n^2 p.
$$
Observe that the second 
inequality implies the first and simplifies to
$$
\sqrt{\frac{p-1}{2p}}f < n.
$$
Since $p$ is greater than $3$ and divides $f$, it is 
possible to choose $n$ to satisfy
$$
\sqrt{\frac{p-1}{2p}}f < n < \frac{p-1}{p} f.
$$  
Since such $n$ automatically satisfies 
the inequality (\ref{last}), we conclude that 
$W$ has maximal non-comparable types at $[\alpha_K], [\beta_L]$ but 
that the Jordan type of $\alpha_K^*(W_K \otimes_K W_K)$ 
is strictly greater than that of $\beta_L^*(W_L \otimes_L W_L)$. 
Thus, $\beta_L$ has maximal Jordan type on $W$ but not on $W \otimes W$.

\end{ex}

Our second example shows that the maximal Jordan type
of $M\otimes N$ can occur at a 
$\pi$-point at which one of $M, N$ does not have maximal
Jordan type and neither has projective type.  
This phenomenon can only occur if $\Pi(G)$ is reducible.

\begin{ex}
\label{tensor0} 
As in the previous example, 
let $G$ be a finite group with exactly two conjugacy classes of maximal
elementary abelian $p$-groups (e.g., the $p$-Sylow subgroup of the
wreath product $\Z/p\Z \, \wr \, S_p$).
Write $\Pi(G) = X \cup Y$ with $X, ~ Y$ irreducible closed subsets, and let 
$[\alpha_K] \in X$, $[\beta_L] \in Y$ be generic points.  Choose 
cohomology classes 
$\zeta, \xi \in \HHH^\bu(G,k)$ such that
$$\alpha_K^*(\zeta_K) = 0 \not= \beta_L^*(\zeta_L), \quad 
\alpha_K^*(\xi_K) \not= 0 = \beta_L^*(\xi_K).$$
Let $L_\zeta$ (respectively, $L_\xi$) be the corresponding Carlson module 
so that  the support of $L_\zeta$ (resp., $L_\xi$) is the zero locus of $\zeta$
(resp., $\xi$).  Recall that the Jordan type of $L_\zeta$ on its support 
is $[\text{proj}] + 1[p-1] + 1[1]$ for some $m$, and 
similarly for $L_\xi$ \cite[4.8]{FPS}.
Let $M = L_\zeta \oplus L_\xi^{\oplus 2}$, so that $M$ has maximal
Jordan type $[\text{proj}] + 1[p-1] + 1[1]$ at $\alpha_K$,  and  Jordan type
$[{\rm proj}] + 2[p-1] + 2[1]$ at $\beta_L$.   Similarly, let 
 $N = L_\zeta^{\oplus 2} \oplus L_\xi$, so that $N$ has Jordan type
$[\text{proj}] + 2[p-1] + 2[1]$ at $\alpha_K$  and maximal Jordan type
$[{\rm proj}] + 1[p-1] + 1[1]$ at $\beta_L$.  
 Then $M\otimes N$ has 
 the same Jordan type at both generic points $\alpha_K$ and $\beta_L$, so that this common generic Jordan type is maximal.
 \sloppy
 {

 }

\end{ex}

Corollary \ref{easy} enables us to prove the following property of
the non-maximal support variety.

\begin{prop}
\label{Gamma}
Assume $\Pi(G)$ is irreducible and  let $M, N$ be $kG$-modules. 
Then 
$$\Gamma(G)_{M\otimes N} ~ =  ~ (\Gamma(G)_M \cup \Gamma(G)_N) \cap 
(\Pi(G)_M \cap \Pi(G)_N).$$

\end{prop}

\begin{proof}
If $\xymatrix@-.5pc{\alpha_K: K[t]/t^p \ar[r] & KG}$ 
is a $\pi$-point such that either $\alpha_K^*(M_K)$
or $\alpha_K^*(N_K)$ is projective, then 
$\alpha_K^*(M_K) \otimes_K \alpha_K^*(N_K)$
is projective and thus of maximum type.  Consequently, statement (1) of 
Theorem \ref{max} implies that $\alpha_K^*(M_K \otimes_K N_K)$ is also
projective and thus also maximum.  For such $\alpha_K$,
$[\alpha_K] \notin \Gamma(G)_{M\otimes N}$.  In other words,
$\Gamma(G)_{M\otimes N} ~ \subset ~ \Pi(G)_M \cup \Pi(G)_N$.

Now suppose that $\alpha_K$ 
has maximal Jordan type on both $M$ and $N$.  
The statement (3) of Theorem~\ref{max} implies that 
$\alpha_K$ has maximal Jordan type on $M \otimes N$.  
Hence, $\Gamma(G)_{M\otimes N} ~ \subset ~ 
\Gamma(G)_M \cup \Gamma(G)_N$. We have established the inclusion
\sloppy
{

}

$$\Gamma(G)_{M\otimes N} ~ \subset  ~ (\Gamma(G)_M \cup \Gamma(G)_N) \cap 
(\Pi(G)_M \cap \Pi(G)_N).$$

\vspace{0.15in}

On the other hand,  assume that neither $\alpha_K^*(M_K)$
nor $\alpha_K^*(N_K)$ is projective, and that 
the Jordan type of  either $\alpha_K^*(M_K)$ or $\alpha_K^*(N_K)$ is
not maximal.  In other words, we assume that 
$[\alpha_K] \in (\Gamma(G)_M \cup \Gamma(G)_N) \cap 
(\Pi(G)_M \cap \Pi(G)_N)$.  
Corollary~\ref{easy} implies that $\alpha_K$ 
does not have  maximal Jordan type on $M \otimes N$.   
Hence, $[\alpha_K] \in \Gamma(G)_{M\otimes N}$.
\end{proof}

\begin{remark}
The previous examples show the necessity of the hypothesis of
irreducibility in Proposition \ref{Gamma}.  Example \ref{tensor0}
contradicts the inclusion 
$$\Gamma(G)_{M\otimes N} ~ \supset  ~ (\Gamma(G)_M \cup \Gamma(G)_N) \cap 
(\Pi(G)_M \cap \Pi(G)_N),$$
whereas Example \ref{tensor1}  contradicts the inclusion
$$\Gamma(G)_{M\otimes N} ~ \subset  ~ (\Gamma(G)_M \cup \Gamma(G)_N) \cap 
(\Pi(G)_M \cap \Pi(G)_N).$$

\end{remark}

Next we offer a suggestive characterization of modules 
of constant Jordan
type for those finite group schemes $G$ with $\Pi(G)$ irreducible.

\begin{prop}  \label{const-criterion}
If $\Pi(G)$ is irreducible, then a non-projective $kG$-module 
$M$ has constant Jordan type 
if and only if for every finite dimensional $kG$-module $N$ the 
tensor product $M\otimes N$ has the property that 
$\Gamma(G)_N = \Gamma(G)_{M\otimes N}$.
\end{prop}

\begin{proof}
First assume that $M$ does not have constant Jordan type.  If $N = k$
(the trivial module), then $\Gamma(G)_k$ is empty 
whereas $\Gamma(G)_{k\otimes M}$
is not empty.

Conversely, assume that $M$ has constant Jordan type and $M$ is not
projective.  Let $\xymatrix@-.5pc{\alpha_K: 
K[t]/t^p \ar[r] & KG}$ be a $\pi$-point 
 such that $N$ has maximal Jordan type at $\alpha_K$.    
By Theorem~\ref{max}, the Jordan type of 
$\alpha^*_K(M_K\otimes N_K) \simeq \alpha^*_K(M_K) 
\otimes \alpha^*_K(N_K)$ is maximal for $M \otimes N$. 
Hence, we get the inclusion 
 $\Gamma(G)_N \supset \Gamma(G)_{M\otimes N}$. 
 To prove the opposite inclusion, assume 
$[\alpha_K] \in \Gamma(G)_N$.   Since $M$ is not projective, 
Corollary~\ref{easy} implies 
that $\alpha_K$ does not have maximal Jordan type on $M \otimes N$.   
Hence, $[\alpha_K] \in \Gamma(G)_{M\otimes N}$, and 
we have established the opposite inclusion. 
\end{proof}


\section{Endotrivial modules}
\label{end}

We recall the definition of an endotrivial module, classically formulated for
finite groups but admitting a natural extension to all finite group schemes.  
Endotrivial modules were introduced by Dade \cite{Da}, who showed that 
for an abelian $p$-group, the only endotrivial $kG$-modules have the form 
$\Omega^n(k) \oplus P$ where $P$ is a projective module. The endotrivial 
modules are the building blocks for the endopermutation modules which 
for many groups are the sources of the simple modules and are also a part
of the Picard group of self equivalences of the stable module category. 
See \cite{CT} for references. A classification of the endotrivial modules 
for finite $p$-groups was completed in \cite{CT3}.

\begin{defn}
Let $G$ be a finite group scheme over $k$.
A $kG$-module $M$ is an endotrivial module provided $\End_k(M)$
is stably isomorphic as a $kG$-module to the trivial module. In other words, 
$M$ is endotrivial provided that there exists a $kG$-projective module $P$
and a $kG$-isomorphism
$$\Hom_k(M,M) ~ \cong ~ k \oplus P.$$
\end{defn}

As can readily be verified (for example, 
using formula (\ref{tensor2}) of the Appendix)  an
indecomposable $k[t]/t^p$-module is 
endotrivial if and only if it is stably isomorphic to either 
$1[1]$ or $1[p-1]$, the trivial module $k$ 
or the Heller shift $\Omega^1(k)$
of $k$.  More generally, Theorem 
\ref{endo} below implies that for any finite
group scheme the Heller shifts $\Omega^i(k)$ 
of the trivial module 
are endotrivial modules.  As mentioned earlier, for elementary 
abelian $p$-groups, these
are the only indecomposable endotrivial modules. 
On the other hand, there
do exist sporadic examples of finite 
groups admitting other endotrivial modules.
For example, if $G$ is a dihedral group of order 8, then 
$\Rad(kG)/\Rad^4(kG)$ is the direct sum of two modules
of dimension three that are endotrivial (see \cite{CT}).

As we show in Theorem \ref{endo} below, 
every endotrivial module is a module
of constant Jordan type.  As seen in section 2 and 
as well in the next section, there exist many 
examples of modules of constant 
Jordan type which are not direct 
sums of endotrivial modules.

If $M, N$ are $kG$-modules, then we may 
identify $\Hom_k(M,N)$ as a $kG$-module
with $M^\# \otimes N$, where $M^\# = \Hom_k(M,k)$.   
For our purposes,
it suffices to analyze $M^\#$ and then 
apply Section \ref{behave} in order to 
investigate $\Hom_k(M,N)$.

Let $G$ be a finite group scheme and (as usual) 
let $kG$ denote the group algebra
of $G$.  Denote by $S$ the antipode of the Hopf 
algebra $kG$.  If 
$\xymatrix@-.5pc{\rho_M: kG \ar[r] & 
\End_k(M)} \simeq M^\# \otimes M$ is a finite 
dimensional representation of $G$ determining
the $kG$-module $M$, then
\begin{equation}
\label{tr}
\rho_{M^\#} ~ = ~ \phi \circ \rho_M \circ S: 
~ \xymatrix{kG \ar[r]& kG \ar[r] & M^\# \otimes M \ar[r] & M\otimes M^\#},
\end{equation}
where $\phi$ exchanges factors (and thus is
the transpose from the point of view of matrices).

Observe that the dual $[i]^\#$ of the indecomposable 
$k[t]/t^p$-module $[i]$ is indecomposable, and thus 
isomorphic to $[i]$ as can be seen by comparing 
dimensions.    The following 
proposition enables us to work with the 
generic and maximal Jordan types of
$\Hom_k(M,N)$ for finite dimensional $kG$-modules $M, ~ N$.

\begin{prop}
\label{dual}
Let $G$ be a finite group scheme, $M$ be a 
finite dimensional $kG$-module, and 
$\xymatrix@-.5pc{\alpha_K: K[t]/t^p \ar[r] & KG}$ be a $\pi$-point 
of $G$.  Assume that 
$\alpha_K$ has maximal Jordan type on $M$.  Then
$$
\alpha_K^*(M_K^\#) ~ \simeq  ~ \alpha_K^*(M_K) ~ \simeq
(\alpha_K^*(M_K))^\#.
$$
Moreover, $\alpha_K$ has maximal Jordan type on $M$ if and
only if it has maximal Jordan type on $M^\#$.
\end{prop}

\begin{proof}
Let $C_K \subset G_K$ be a unipotent abelian 
group scheme through which $\alpha_K$
factors, so that $\alpha_K$ has maximal Jordan type for 
$M_K$ as a $KC_K$-module.   The restriction 
of the $KG$-module $M_K^\#$ to $KC_K$ is the 
dual of the restriction of $M_K$, since
$\xymatrix@-.7pc{KC_K \ar[r] & KG_K}$ is a map of Hopf algebras.  
Thus, to prove the first asserted 
isomorphism, it suffices to assume $G$ is a 
unipotent abelian finite group 
scheme.  

Extending scalars if necessary so that $k$ is perfect, we get 

 $kG \simeq k[T_1, \dots, T_r]/(T_i^{p^{e_i}})$ (see \cite[14.4]{W}).   
Let $t_i = T_i^{p^{e_i-1}}$. Let $I = (T_1, \dots, T_r)$ be the 
augmentation ideal of $kG$, and let $I^{(p)}$ 
be the ideal generated by $(t_1, \dots, t_r)$.
Let $\nabla: kG \to kG \otimes kG$ be the coproduct in $kG$.   Recall that 
$\nabla(\alpha_K(t)) - 1 \otimes \alpha_K(t) - \alpha_K(t) \otimes 1 
\in  I \otimes I$ (\cite[I.2.4]{J}). 
Since $\alpha_K(t)$ has $p$-th power 0, 
we can refine this further, concluding that  
\begin{equation}
\label{nabla1}
\nabla(\alpha_K(t)) - 1 \otimes \alpha_K(t) - \alpha_K(t) \otimes 1 
\in  I \otimes I^{(p)} + I^{(p)} \otimes I.
\end{equation} 
Let $\mu:kG\otimes kG \to kG$ 
be the multiplication map in $kG$.   We have 
$(\mu \circ (S \otimes \rm Id))(\nabla(\alpha_K(t)) = 
\epsilon(\alpha_K(t)) = 0$  
by one of the Hopf algebra axioms (see \cite[I.2.3]{J}) where 
$\epsilon: kG \to k$ is the counit map, and $S$ is the antipode.   
Hence, applying
$\mu \circ (S \otimes \rm Id)$ to (\ref{nabla1}), we get  
$$
\alpha_K(t) + S(\alpha_K(t)) \subset I \cdot I^{(p)} + I^{(p)} \cdot I.
$$
Moreover, we obtain the same inclusion if we change the base field from $K$ to
any field extension $L$ over $K$.
Therefore, we can apply \cite[1.12]{FPS} to conclude that $\rho_M(\alpha_K(t))$
and $\rho_M(S(\alpha_K(t)))$ have the same (maximal) Jordan type.  Hence,  
$\rho_M(\alpha_K(t))$ and $\rho_{M^\#}(\alpha_K(t)) = 
\phi(\rho_M(S(\alpha_K(t))))$ 
have the same (maximal) Jordan type.  We conclude that 
$\alpha_K^*(M_K^\#) ~ \simeq ~ \alpha_K^*(M_K)$.

The second isomorphism follows immediately 
from the observation that $[i]^\# = [i]$.

\end{proof}

Since sending $M$ to $M^\#$ is idempotent, 
we get the following corollary.

\begin{cor}
Let $G$ be a finite group scheme, and let $M$ 
be a finite dimensional $kG$-module.
Then 
$$\Gamma(G)_M ~ = ~ \Gamma(G)_{M^\#}.$$
\end{cor}

The next corollary follows 
immediately from Corollary \ref{tensor}, Proposition \ref{dual}
and the isomorphism 
 $\Hom_k(M,N) \simeq M^\# \otimes N$. 

\begin{cor}
\label{const}
Let $G$ be a finite group scheme, and let $M, ~N$ 
be finite dimensional $kG$-modules
of constant Jordan types $\ul a = a_p[p] + 
\cdots + a_1[1], ~ \ul b = b_p[p] + \cdots + b_1[1]$, respectively.  
Then $Hom_k(M,N)$ has constant Jordan type 
$\ul a \otimes \ul b$ (given explicitly 
by the formula in the Appendix).
\end{cor}

\begin{cor}
\label{hom}
Let $G$ be a finite group scheme, and consider two finite 
dimensional $kG$-modules $M$, $N$, and a $\pi$-point $\alpha_K$ of $G$. 
If $\alpha_K$ has maximal Jordan type on $Hom(M,N)$,   then  
$$\alpha_K^*(\Hom_K(M_K,N_K)) 
 ~ \simeq ~ \Hom_K(\alpha_K^*(M_K),  
\alpha_K^*(N_K))$$
\end{cor}

\begin{proof}
Let $i: C_K \hookrightarrow G_K$ be a unipotent 
abelian group scheme through which $\alpha_K$ factors, so that $\alpha_K$ 
is maximal on $\Hom_K(M_K,N_K)$ as a $KC_K$-module.  
Since $i: C_K \hookrightarrow G_K$ is a map of 
Hopf algebras, it commutes with $\Hom$.   
Hence, we may assume  that $G$ is a unipotent 
abelian group scheme.  In particular, $\Pi(G)$ is irreducible. 

Since $\alpha_K$  
has maximal Jordan type on $\Hom(M,N) \simeq  M^\# \otimes N$,   
Theorem \ref{max}(2) implies that
$$
\alpha_K^*(M_K^\# \otimes N_K) \simeq \alpha_K^*(M_K^\#) 
\otimes \alpha^*_K(N_K).
$$  
If $\alpha_K^*(M_K^\#) \otimes \, \alpha^*_K(N_K)$ 
is projective, then either 
$\alpha^*_K(N_K)$ or $\alpha^*_K(M_K^\#)$ is projective. 
Since projectivity of $\alpha^*_K(M_K^\#)$ 
implies projectivity of $\alpha^*_K(M_K^\#)$, we conclude that in this case 
$\Hom_K(\alpha_K^*(M_K),  \alpha_K^*(N_K))$ is projective. 

Assume that $\alpha_K^*(M_K^\#) \otimes \, \alpha^*_K(N_K)$ 
is not projective. In this case neither 
$\alpha^*_K(N_K)$ nor $\alpha^*_K(M_K^\#)$ is projective. 
Since $\Pi(G)$ is irreducible, Corollary~\ref{easy} implies 
that $\alpha_K$  
has maximal Jordan types on both $N$ and $M^\#$.  Hence, 
$\alpha^*_K(M_K^\#) ~\simeq ~ (\alpha^*_K(M_K))^\#$. 
Therefore, $\Hom_K(\alpha_K^*(M_K),  
\alpha_K^*(N_K)) \simeq (\alpha^*_K(M_K))^\# \otimes 
\alpha_K^*(N_K) \simeq \alpha^*_K(M_K^\#) \otimes \alpha_K^*(N_K) 
\simeq \alpha_K^*(\Hom_K(M_K,N_K))$.
{\sloppy

}
\end{proof}

We now conclude that endotrivial modules are modules
of constant Jordan type. The second statement of the 
theorem provides a ``local" criterion of endotriviality, 
similar to the projectivity criterion given  by the Dade's lemma
(see \cite{Da}).

\begin{thm}
\label{endo}
Let $G$ be a finite group scheme, and let $M$ 
be a finite dimensional $kG$-module.
\begin{enumerate}
\item 
If $M$ is endotrivial, then $M$ has constant Jordan type 
of the form either $m[p] + 1[1]$ or $m[p]+1[p-1]$ for some 
$m \geq 0$, and thus $\alpha_K^*(M_K)$ is endotrivial
for every $\pi$-point $\alpha_K$ of $G$.
\item 
Conversely, if $\alpha_K^*(M_K)$ is endotrivial for 
each $\pi$-point $\alpha_K$
of $G$ (and hence of the form either
$m[p] + 1[1]$ or $m[p]+1[p-1]$), then $M$ is endotrivial.
\end{enumerate}
\end{thm}

\begin{proof}
Observe that any endotrivial module must have 
dimension whose square is congruent
to 1 modulo $p$ and thus must have dimension 
congruent to either 1 or $p-1$ modulo
$p$.   To prove the first assertion, we assume that $M$ 
is endotrivial, so  that $\End_k(M) = k \oplus \proj$.
Thus,  $\alpha_K^*(\End_K(M_K))$ has Jordan type  $m[p] + 1[1]$ 
at each $\pi$-point $\alpha_K$.   In particular, 
every $\pi$-point $\alpha_K$ has maximal Jordan type on $\End_k(M)$. 
By Corollary~\ref{hom}, 
$\alpha_K^*(\End_K(M_K)) \simeq \End_K(\alpha^*_K(M_K))$. 
Hence, $\alpha^*_K(M_K)$ is an endotrivial $K[t]/t^p$-module.    
The statement now 
follows from the fact that the only such 
modules are of the form  $m[p] + 1[1]$ or $m[p] + 1[p-1]$ 
for some  $m \geq 0 $.

To prove the converse,  we
assume that $\alpha_K^*(M_K)$ is 
an endotrivial $K[t]/t^p$-module
for each $\pi$-point $\xymatrix@-.5pc{\alpha_K: K[t]/t^p \ar[r] & KG}$.  
Thus, for each $\alpha_K$,
 $\alpha_K^*(M_K)$ has Jordan type either  
$m[p] + 1[1]$ or 
$m[p] + 1[p-1]$ for some $m \geq 0$.  
Since the dimension of 
 $M$ can not be congruent
to both $1$ and $p-1$ modulo $p$, 
we conclude that $M$ has constant Jordan
type.  Consider the short exact   sequence
\begin{equation}
\label{trace}
\xymatrix{
0 \ar[r] & X \ar[r] & \End_k(M) \ar[r]^{\quad \quad  {\rm Tr}} &  k \ar[r] & 0,
}
\end{equation}
where $\rm Tr$ is the trace map.  By Corollary 
\ref{hom}, $\alpha_K^*(\End_K(M_K)) \simeq
\End_K(\alpha_K^*(M_K))$.  Moreover, 
because the dimension of $\End_k(M)$ 
is relatively prime to $p$, the trace 
map of (\ref{trace}) splits.  
Pulling back the split short exact 
sequence (\ref{trace})
via $\alpha_K$, we conclude that 
$$
\alpha_K^*(X) \simeq \ker\{\End_K(\xymatrix{\alpha_K^*(M_K)) \ar[r] & K}\}  
\simeq m[p]
$$
is projective
for all $\pi$-points $\alpha_K$ and 
thus $X$ is projective (\cite[5.4]{FP2}).  Hence, $M$ is
endotrivial.
\end{proof}

\vskip .2in


\section{Constructing modules of constant Jordan type}
\label{constr}

In this section, we consider two different methods of constructing modules
of constant Jordan type.  Proposition \ref{extend1} presents the observation
that an extension of modules of constant Jordan type has ``total module"
also of constant Jordan type if the extension splits when pulled back 
along any $\pi$-point.   This observation fits well with the Auslander-Reiten
theory of almost split exact sequences 
as we discuss in \S \ref{AR}.   Proposition 
\ref{create1} presents a method of constructing extensions of constant 
Jordan type whose pull-backs along $\pi$-points are not split.

We shall frequently utilize the Tate cohomology ring 
$\widehat \HHH^*(G,k)$ for a finite group scheme $G$, and 
The Tate Ext groups $\widehat \Ext^*(M,N)$ for $G$-modules $M,N$.  
Tate cohomology ring $\widehat \HHH^*(G,k)$ is a graded 
commutative $k$-algebra which coincides with regular cohomology 
in positive degrees.  One advantage of 
Tate cohomology that we will  exploit throughout this section is 
that it allows for arbitrary degree shifts.   
Namely, we have the following formulas  
$$\widehat \Ext^n(M,N) \simeq \widehat \Ext^0(\Omega^{n} M, N) 
= \Hom_{\stmod (kG)}(\Omega^{n} M, N) \simeq $$
$$\Hom_{\stmod (kG)}(M, \Omega^{-n} N) = \widehat \Ext^0 (M, \Omega^{-n} N)$$  
for any $n \in \Z$.  We refer the reader to \cite{B} for  further details.  

\begin{prop} \label{extend1}
Suppose that $G$ is a finite group scheme over $k$. Let $M$ and 
$N$ be $kG$-modules of constant Jordan type, and suppose that 
\begin{equation}
\label{seq}
\xymatrix{
 0 \ar[r] & M \ar[r] & B \ar[r] & N \ar[r] & 0 
}
\end{equation}
is an exact sequence. Let $\zeta \in \Ext_{kG}^1(N,M)$ be the 
class of  (\ref{seq}).
If for every $\pi$-point
$\xymatrix@-.5pc{\alpha_K: K[t]/t^p \ar[r] & KG}$ 
the restriction 
$\alpha_K^*(\zeta)$  is zero, then $B$ has constant Jordan type.
Moreover, if the Jordan types of $M$ and $N$ are $\sum_{i=1}^p
m_i[i]$ and $\sum_{i=1}^p n_i[i]$, then the Jordan type of 
$B$ is $\sum_{i=1}^p (m_i+n_i)[i]$. 
\end{prop}

\begin{proof}
Because the cohomology class vanishes under restriction to a 
$\pi$-point  $\alpha_K$, we have that the restriction of $(\ref{seq})$ along
$\alpha_K$ splits and $\alpha_K^*(B_K) \simeq 
\alpha_K^*(M_K) \oplus \alpha_K^*(N_K)$. The result is now obvious.
\end{proof}

Proposition \ref{extend1} does not always produce ``new" examples
of modules of constant Jordan type as we observe in the 
special case  $G = \Z/2\Z \times \Z/2\Z$.

\begin{ex}
Let $G$ be $\Z/2\Z \times \Z/2\Z$.  As presented in \cite{Ba}, \cite{HR}, 
there is a complete classification of the $kG$-modules.  Using this
classification, we observe that the only indecomposable $kG$-modules 
of constant Jordan type are of the form $\Omega^n(k)$ for some $n$. 

Namely, it is shown in \cite{HR} that all of the indecomposable 
$kG$-modules of odd dimension are of the form $\Omega^n(k)$ for some
$n$.  On the other hand, the non-projective indecomposable
modules of even dimension are all isomorphic to $L_{\zeta}$ for 
$\zeta \in \HHH^n(G,k)$.  The support varieties of these even dimensional
modules are proper non-trivial subvarieties of $\Pi(G)$, so that none of 
these have constant Jordan type.  

It is instructive to look more closely to see why Proposition \ref{extend1}
does not determine other modules of constant Jordan type in this
example.  Observe that  $\HHH^*(G,k) \cong \Ext_{kG}^*(k,k)$ is a 
polynomial ring in two variables having no non-zero element whose restriction
to every $\pi$-point vanishes. Hence the only possible application 
of Proposition \ref{extend1} would be in a situation where $N \cong 
\Omega^n(k)$, $M \cong \Omega^m(k)$ and $n < m$. Then (\ref{seq})
represents an element of negative Tate cohomology, 
$\zeta$ of $\widehat{\HHH}^{n-m+1}(G,k)$.   By Proposition \ref{negcoho}
which follows,
$\zeta$ restricts to zero at every $\pi$-point, but
the middle term of this non-split short exact sequence splits 
as $\Omega^{n+a}(k) \oplus 
\Omega^{n+b} \oplus \proj$ where $a$ and $b$ are nonnegative integers
such that $a+b = m-n$ (see the proof of Theorem \ref{not-endo}).
\end{ex}

The preceding example is special since $\Z/2\Z \times \Z/2\Z$ has tame 
representation type.   For more general groups, the observation of
Proposition \ref{extend1} in conjunction with the following proposition
does give new examples.

\begin{prop} 
\label{negcoho}  
Let $G$ be a finite group scheme with the
property that every $\pi$-point factors by way of a flat map
through a unipotent abelian group scheme
whose cohomology has Krull dimension at least 2.   Let $\zeta \in
\widehat{\HHH}^n(G,k)$ for $n < 0$ be an element in negative Tate cohomology
of $G$ corresponding to a short exact sequence of the form
\begin{equation}
\label{seqq}
\xymatrix{
 0 \ar[r] & k \ar[r] & E \ar[r] & \Omega^{n-1}(k) \ar[r] & 0.
}
\end{equation}
Then for any $\pi$-point 
$\xymatrix@-.5pc{\alpha_K: K[t]/t^p \ar[r] & KG}$,  the 
pull-back of (\ref{seqq}) along $\alpha_K$ is split
(i.e., $\alpha_K^*(\zeta_K) = 0$).
\end{prop}

\begin{proof}
Recall that every element $\zeta$ of $\widehat{\HHH}^n(G,k)$ is represented 
by a homomorphism $\xymatrix{\zeta^\prime : k \ar[r] & \Omega^{-n}(k)}$. In 
order to prove the lemma, it is sufficient to show that for any 
$\pi$-point $\xymatrix{\alpha_K:K[t]/(t^p) \ar[r] & KG}$, the 
restriction of $\zeta^\prime$ factors through a projective $K[t]/(t^p)$-module.
We proceed to establish such a factorization using our knowledge
of the modules $\Omega^{-n}(K)$ for $n < 0$. 
By Theorem \ref{iso}, equivalent $\pi$-points induce the same
map in cohomology.   Thus, it suffices to prove the statement for some
representative  $\alpha_K$ of each $[\alpha_K] \in \Pi(G)$.

Our hypothesis immediately allows us to replace $G$ by some abelian unipotent
group scheme (defined over $K$) whose cohomology has Krull dimension at 
least 2.  After possibly passing to some finite extension of $K$, 
\cite[14.1]{W} enables us to conclude that 
$KG \simeq  K[T_1, \ldots, T_r]/(T_1^{p^{e_1}}, \dots, T_r^{p^{e_r}} )$, 
where $r \geq 2$.  
Because neither our
hypothesis nor our conclusion depends upon the coalgebra structure on $KG$,
we may assume that $G$ is an abelian $p$-group.  Let $t_i = T_i^{p^{e_i-1}}$
and let $E \subset G$ be the (unique) elementary abelian $p$-group
with
$KE = k[t_1,\ldots , t_r]/(t_1^p, \dots, t_r^p) \subset KG$.  
By \cite[4.1]{FP1}, any
$K$-rational $\pi$-point of $KG$ has a representative
factoring through $KE$.   Thus, we may assume that $G = E$ is an elementary
abelian $p$-group of rank at least 2.  Changing the generators of $KE$, we
may further assume $t_1 = \alpha_K(t)$; 
moreover, it suffices to assume that $E$
has rank 2, for if the proposition is valid for 
an elementary abelian subgroup of 
$G$ of rank 2, then it is valid for $G$ itself. 
Thus, we are reduced to the case
that $KG$ is isomorphic to $K[u,v]/(u^p,v^p)$ with $u = \alpha_K(t)$. 

The structure of a
minimal $KG$-projective resolution $P_\bu \to K$ is well known \cite{CTVZ}, 
with $P_{m-1} = KG^{\times m}$. A set of generators $a_1, \dots, 
a_m$ for $P_{m-1}$ can be chosen so that $\Omega^m(K)$ 
is the submodule generated by the elements 
$$
u^{p-1}a_1, \ \ va_1 -ua_2, \ \ v^{p-1}a_2- u^{p-1}a_3, \ \ 
va_3-ua_4, \ \ \ldots, \ \ va_{m-1} -ua_m, \ \ v^{p-1}a_m$$
for  $m$  even and 
$$ua_1, \ \ va_1-u^{p-1}a_2, \ \ v^{p-1}a_2 - ua_3, \ \ va_3  - u^{p-1}a_4, \ \ \ldots, \ \ v^{p-1}a_{m-1} - ua_m, \ \ va_m $$
for $m$ odd. 
Every $KG$-fixed point of $P_{m-1}$, and thus also of $\Omega^m(K)$ is a 
linear combination of the elements $u^{p-1}v^{p-1}a_i$.  
Moreover, for $i$ odd and $m$ even,
$$
u^{p-1}v^{p-1}a_i \ = \ u^{p-1}v^{p-2}(va_i - ua_{i+1})
$$
and for $i$ even and $m$ even,
$$
u^{p-1}v^{p-1}a_i \ = \ u^{p-1}(v^{p-1}a_i - u^{p-1}a_{i+1}).
$$
Similar formulas hold for $m$ odd. 
Consequently, every $KG$-fixed point of $\Omega^m(K)$ has the 
form $u^{p-1}z$ for some $z \in \Omega^m(K)$.    Thus,
as a map of $K[t]/t^p$-modules (via 
$\xymatrix@-.5pc{\alpha_K: K[t]/t^p \ar[r] & KG}$), $\zeta_K^\prime$
factors as the composition of 
$\xymatrix@-.5pc{K \ar[r] & K[t]/t^p}$ sending $1$ to $t^{p-1}$ and
$\xymatrix@-.5pc{K[t]/t^p \ar[r] &  \Omega^m(K)}$ sending $1$ to $z$.
Hence, $\alpha_K^*(\zeta^\prime)$ factors through the rank 1
projective module $K[t]/t^p$ as required.
\end{proof}

The following property is an immediate corollary  of Proposition~\ref{negcoho}.

\begin{cor}
\label{syzygy}
Let $E$ be an elementary abelian $p$-group of rank at least 2, and   
let $$\xymatrix{\theta: \Omega^m(k) \ar[r] & \Omega^n(k)} $$ 
be a homomorphism. 
If the restriction of $\theta$ along some $\pi$-point 
$\alpha_K: K[t]/t^p \to KE$  does not factor through a projective 
$K[t]/t^p$-module, then $m \geq n$.

\end{cor}

The following proposition shows that there are limitations on the
Jordan types which can be realized as extensions.

\begin{prop}  
\label{extend2}
Let $G$ be a finite group scheme over a field $k$ 
of characteristic $p>2$ with the
property that every $\pi$-point factors by way of a flat map
through a unipotent abelian group scheme
whose cohomology has Krull dimension at least 2. 
Then there does not exist a $kG$-module $M$ of constant
Jordan type $n[p] + 1[2]$ which is an extension of a $kG$-module of
constant Jordan type $m[p] + 1[1]$ and one of constant Jordan type
$(n-m)[p] + 1[1]$. 
\end{prop}

\begin{proof}
As argued above in the proof of Proposition \ref{negcoho}, we may
assume that $G$ is an elementary abelian $p$-group of rank 2.
Consider a short exact sequence of the form
\begin{equation}
\label{seqqq}
\xymatrix{
E: \qquad 0 \ar[r] & L \ar[r] & M \ar[r] & N \ar[r] & 0
}
\end{equation}
in which both $L$ and $N$ have stable constant Jordan type $1[1]$.
 By Theorem \ref{endo}, $L$ and $N$ are endotrivial 
modules, so that 
$L \cong \Omega^a(k)$, $N \cong \Omega^b(k)$
for some integers $a$ and $b$. 
Since $p>2$, both $a$ and $b$ must be even integers. 

Thus the class
of $(\ref{seqqq})$ is a cohomology class $\zeta$ in 
$$ 
\Ext_{kG}^1(\Omega^b(k), \Omega^a(k)) \cong \widehat{\HHH}^{b-a+1}(G,k).
$$
If $M$ has constant Jordan type $n[p] + 1[2]$, then the restriction
of (\ref{seqqq}) along any $\pi$-point $\alpha_K$ does not split;
equivalently, the restricted class $\alpha_K^*(\zeta)$
does not vanish. However, the only such cohomology classes $\zeta$
are scalar multiples of the identity in degree 
$b-a+1 = 0$.   Since both $a$ and $b$ are even integers, such
classes $\zeta$ can not occur. 
\end{proof}

We now proceed to describe a second method of constructing 
modules of stable constant Jordan
type $n[1]$ that cannot be created by the methods of Proposition \ref{extend1}.
The construction can be summarized as follows: if a 
certain map represented by a matrix  with coefficients in 
$\widehat \HHH^*(G,k)$
has the maximal possible rank when restricted to any $\pi$-point of $G$, 
then it has a kernel of constant Jordan type.

\begin{thm}
 \label{create1}
Let $G$ be a finite group scheme, and choose integers  
$m > n$, $m_i$ and $n_j$ such that 
all $m_i$ and $n_j$ are even if $p > 2$. 
Choose cohomology classes $\zeta_{i,j} \in \widehat\HHH^{m_j-n_i}(G,k)$ and let
$\hat \zeta_{i,j}: \Omega^{m_j}(k) \to \Omega^{n_i}(k)$ 
represent $\zeta_{i,j}$.
We consider an exact sequence
$$
\xymatrix{
0 \ar[r] & L \ar[r] & M \ar[rr]^{\varphi = 
(\hat \zeta_{i,j})} && N \ar[r] & 0
}
$$
of $kG$-modules
where
$$
M \ \cong \sum_{j = 1}^m \Omega^{m_j}(k) \oplus \proj
\quad \text{and} \quad N \cong \sum_{i = 1}^n \Omega^{n_i}(k).
$$
Assume that for every $\pi$-point
$\xymatrix{\alpha_K: K[t]/t^p \ar[r] & KG}$,
the restriction of the matrix of cohomology elements
$(\alpha_K^*(\zeta_{i,j})) \in M_{n,m}(\widehat\HHH^*(K[t]/t^p,K))$
has rank $n$.  Then the module $L$ has stable constant Jordan type 
$(m-n)[1]$.
\end{thm}

\begin{proof}
Let $\xymatrix{\alpha_K: K[t]/t^p \ar[r] & KG}$ be a 
$\pi$-point. For any $s$, which is even if $p > 2$, 
the restriction of the module $\Omega^s(k)$ 
along $\alpha_K$ has the form 
$\alpha_K^*(\Omega^s(K)) \cong K \oplus \proj$. Moreover, 
if we have a map $\xymatrix{\zeta: \Omega^s(k) \ar[r] & \Omega^t(k)}$ 
whose cohomology class 
when restricted along $\alpha_K$ is not zero, then the composition 
$$
\xymatrix{
K \ar[r]^{\iota \qquad} & \alpha_K^*(\Omega^s(K)) 
\ar[rr]^{\alpha_K^*(\zeta_K)} && 
\alpha_K(\Omega^t(K)) \ar[r]^{\qquad \rho} & K,
}
$$ 
where $\iota$ and $\rho$ are the split inclusion and projection 
maps, is an isomorphism. Thus our hypothesis asserts 
that the composition
$$
\xymatrix{
K^m \ar[r]^{\iota \quad} & \alpha_K^*(M_K) \ar[rr]^{\alpha^*(\varphi_K)} && 
\alpha_K^*(N_K) \ar[r]^{\quad \rho} & K^n
}
$$ 
has rank $n$.   Since $\alpha_K^*(M_K) = \proj +m[1]$ and
$\alpha_K^*(N_K) = \proj +n[1]$, it
follows that the kernel of $\varphi$, $\alpha_K^*(L)$, has
stable Jordan type $(m-n)[1]$. As this happens for any $\alpha_K$, 
we are done. 
\end{proof}

Theorem \ref{create1} is stated for even dimensional cohomology classes
(for $p > 2$), yet as we observe in the following proposition a very similar
argument can also be applied to odd dimensional classes to yield
additional modules of constant Jordan type.

\begin{prop}
\label{create2}

Let $G$ be a finite group scheme, and choose positive integers $m > n$,
$m_1,\ldots m_m$ all odd,  $n_1,\ldots,n_n$ all even. Assume $p>2$. 
Choose cohomology classes $\zeta_{i,j} \in \widehat\HHH^{m_j-n_i}(G,k)$ and let
$\hat \zeta_{i,j}: \Omega^{m_j}(k) \to \Omega^{n_i}(k)$ 
represent $\zeta_{i,j}$.
We consider an exact sequence
$$
\xymatrix{
0 \ar[r] & L \ar[rr] && M \ar[rr]^{\varphi = 
(\hat \zeta_{i,j})} && N \ar[r] & 0
}
$$
of $kG$-modules
where
$$
M \ \cong \sum_{j = 1}^m \Omega^{m_j}(k) \oplus \proj
\quad \text{and} \quad N \cong \sum_{i = 1}^n \Omega^{n_i}(k).
$$
Assume that for every $\pi$-point
$\xymatrix{\alpha_K: K[t]/(t^p) \ar[r] & KG}$,
the restriction of the matrix of cohomology elements
$\alpha_K^*(\zeta_{i,j}) \in M_{n,m}(\widehat\HHH^*(K[t]/t^p,K))$
has rank $n$.  Then the module $L$ has stable constant
Jordan type $(m-n)[p-1] + n[p-2]$.
\end{prop}

\begin{proof}
Our hypothesis implies that the restriction of 
$(\hat \zeta_{i,j})$ via any $\alpha_K$
is a map of $K[t]/t^p$-modules which remains surjective after free
summands are dropped.  We write such a map of $K[t]/t^p$-modules
symbolically as a map of their (stable) Jordan types
$m[p-1] \twoheadrightarrow n[1]$.  Such a surjective map of
$K[t]/t^p$-modules with indicated stable Jordan type necessarily
has kernel with Jordan type $(m-n)[p-1] + n[p-2]$.
\end{proof}

In general, it is very easy to construct examples for which Theorem
\ref{create1} and Proposition \ref{create2} are relevant.  These 
modules are multi-parametrized versions of important modules
introduced and studied by the first author.

\begin{ex}
\label{gencarl}
Let $\xi_1. \dots, \xi_r$ be homogeneous elements in $\HHH^\bu(G,k)$ 
such that  the radical of the ideal generated by the $\xi_i$'s is the 
augmentation ideal of $\HHH^\bu(G,k)$.   Alternatively, for $p$ odd,
let $\xi_1. \dots, \xi_r$ be homogeneous elements in $\HHH^*(G,k)$
of odd degree whose Bocksteins generate an ideal of  $\HHH^\bu(G,k)$
whose radical is the augmentation ideal.
For each $i$, let $\xymatrix{\hat\xi_i: \Omega^{n_i}(k)
\ar[r] & k}$ be a cocycle representing $\xi_i$, where $n_i$ is the 
degree of $\xi_i$.  Define 
$$\varphi \, : \,  \bigoplus_{i=1}^r \Omega^{n_i}(k)
\to k
$$  by the formula 
$\varphi(a_1, \dots, a_r) = \hat\xi_1(a_1) + \dots + \hat\xi_r(a_r)$. 
The condition that either the $\xi_i$'s or the Bockstein's of the $\xi_i$'s
generate an ideal of  $\HHH^\bu(G,k)$
whose radical is the augmentation ideal implies that for each $\pi$-point
$\alpha_K$, there is some $i$ such that $\alpha_K^*(\hat\xi) \not= 0$.
(Namely, no point $[\alpha_K] \in \Pi(G) \cong \Proj \HHH^\bu(G,k)$ lies in
the union of the zero loci of either the $\xi_i$'s or the Bockstein's of 
the $\xi_i$'s.)  Thus,
\begin{equation}
\label{phi}
\xymatrix{
L_{\xi_1, \dots, \xi_r} ~ \equiv
~\Ker \varphi 
}
\end{equation}
has constant Jordan type.
\end{ex}

In Theorem \ref{not-endo} below, we give an explicit example of
such an $L_{\xi_1, \dots, \xi_r}$ which can not be constructed 
using the technique of Proposition \ref{extend1}.  The detailed
verification of this example will occupy the remainder of this
section, and involves the following four lemmas.  

 Recall that if 
$E \simeq (\Z/p\Z)^{\times r}$ is an elementary 
abelian $p$-group of rank $r$, then 
\begin{equation}
\label{orderp}
\HHH^*(E,k) \cong 
\begin{cases} k[\zeta_1, \dots, \zeta_r] \otimes  
\Lambda(\eta_1, \dots, \eta_r) 
\quad & \text{if} \quad p > 2 \\
k[ \eta_1, \dots, \eta_r] \quad & \text{if} \quad p=2. \end{cases}
\end{equation} 
where each $\eta_j$ has degree
one and each $\zeta_i$ has degree 2. Here, 
$\Lambda(\eta_1, \dots, \eta_r)$ is the exterior algebra on the generators 
$\eta_1, \dots, \eta_r$.

\begin{lemma} 
\label{dim-omega}
Let $E$ be an elementary abelian $p$-group of rank $r$, for $r>1$. 
For $n > 0$, 
\begin{enumerate}
\item  $\Dim  \HHH^n(E,k) ~ = ~ \binom{n+r-1}{r-1}$. 
\item $\Dim P_n ~ = ~ \binom{n+r-1}{r-1} \cdot p^r$, where $P_n$ is the
$n$-th term of a minimal $kE$-projective resolution of $k$.
\item $\Dim \Omega^n(k) ~ =  ~p^r \cdot a_{r,n}+(-1)^n, ~ n > 0$
where 
$$
a_{r,n} = \binom{n+r-2}{r-1} - \binom{n+r-3}{r-1}
+ \dots + (-1)^{n-1}\binom{r-1}{r-1}. 
$$
\item 
If $r = 2$, then $\Dim(\Omega^{2n}(k)) ~ = ~p^2n+1$.
\item   If $r =3$, then $\Dim(\Omega^{2n}(k)) =  p^3n(n+1) +1$
\end{enumerate} 
Finally, if $n <0$, then $\Dim(\Omega^n(k)) = \Dim(\Omega^{-n}(k)).$
\end{lemma}

\begin{proof}
The computation of $\Dim  \HHH^n(E,k)$ is a straightforward and 
familiar computation.  In a minimal 
projective $kE$-resolution of $k$,
$$
\xymatrix{
\dots \ar[r] & P_2 \ar[r]^{\partial_2} & P_1 \ar[r]^{\partial_1} &
P_0 \ar[r]^{\varepsilon} & k \ar[r] & 0,
}
$$
each $P_n$ is a direct sum of copies of $kE$, the number of copies
equal to the dimension of $\HHH^n(E,k)$. 
Using the vanishing of the Euler
characteristic of an exact sequence, we conclude the asserted
formula for the dimension of $\Omega^n(k)$.
The assertions for $r=2, 3$ are special cases.
 \end{proof}

\begin{lemma}
\label{proj-dim}
Let $E$ be an elementary abelian $p$-group of rank $3$, 
$n$ be a negative integer, and $\xymatrix{Q \ar[r] & \Omega^{2n} (k)}$ 
 be the 
projective cover of the $2n^{\rm th}$ Heller shift of $k$. 
Then $\Dim Q  = p^3(2n^2 -n)$.
\end{lemma} 

\begin{proof}
Taking duals we have an injection
$\xymatrix{\Omega^{-2n} (k) \ar[r] & Q^*}$, 
where $Q^*$ is the injective hull of $\Omega^{-2n} (k)$. 
Hence, $Q^* \simeq P_{-2n-1}$,  the $(-2n-1)^{\rm st}$
term of the minimal projective 
resolution of $k$. As seen in 
Lemma~\ref{dim-omega}(2), the dimension of $P_{-2n-1}$ 
is $ p^3\binom{-2n-1+2}{2} =  p^3(2n^2 -n)$.
\end{proof}

\begin{lemma} \label{syzygy-onto}
Suppose that $E$ is an elementary abelian $p$-group
of rank at least 2.  Let 
$$
\xymatrix{\theta: \Omega^m(k) \ar[r] & \Omega^n(k)}
$$ 
be a homomorphism for some nonnegative integers $m$ and $n$, 
which are both even in case $p>2$. Assume that for some 
$\pi$-point $\xymatrix{\alpha_K:K[t]/(t^p) \ar[r] & KE}$
the restriction along $\alpha$ of $\theta$ does not factor
through a projective $K[t]/(t^p)$-module. Then $\theta$ is surjective.
\end{lemma}

\begin{proof}
By Corollary~\ref{syzygy}, we have $m-n \geq 0$. 
Let $\widehat{\theta}$ denote the 
cohomology class of $\theta$ in 
$$\Ext_{kE}^{m-n}(k,k) \cong \HHH^{m-n}(E,k).
$$  
The condition on the restriction of 
$\theta$ along $\alpha$ together with the  
fact that $m-n$ is even if $p>2$ implies that $\widehat{\theta}$ is a 
non-nilpotent element of $\HHH^{m-n}(E,k)$.  Hence, 
multiplication by $\widehat\theta$ induces an injective map 
$\xymatrix{\HHH^{n}(E,k) \ar[r] & \HHH^m(E,k)}$.  Since
$\Hom_{kE}(\Omega^\ell(k),k)
\cong \HHH^\ell(E,k)$ for $\ell \geq 0$, it follows that 
$\theta$ induces an injective map 
$$
\xymatrix{
\theta^\prime: \Hom_{kE}(\Omega^n(k),k) \ar[r] & 
\Hom_{kE}(\Omega^m(k),k)}.
$$
So $\theta$ must be surjective.
\end{proof}

\begin{lemma}\label{rk2}
Suppose that $E$ is an elementary abelian $p$-group of rank $2$.
Let $\xi_1 \in \HHH^{2m}(E,k)$, $\xi_2 \in \HHH^{2n}(E,k)$, 
and assume that the radical of 
the ideal generated by  $\xi_1, \xi_2$ 
is the augmentation ideal of $\HHH^\bu(E,k)$. 
Consider the exact sequence 
$$
\xymatrix{
0 \ar[r] & L_{\xi_1, \xi_2} \ar[r] & \Omega^{2m}(k) \oplus \Omega^{2n}(k) 
\ar[r]^{\qquad \quad 
\scriptsize{\begin{pmatrix} \xi_1 \\ \xi_2 \end{pmatrix}}} & k \ar[r] &0.
}
$$
Then $L_{\xi_1,\xi_2} \simeq \Omega^{2m+2n}(k)$.
\end{lemma}

\begin{proof}
The condition on $\xi_1$, $\xi_2$ is equivalent 
to the condition that  the matrix 
$(\xi_1, \xi_2)$
has rank 1 when restricted to any $\pi$-point of $G$.  
Hence, by Theorem~\ref{create1}, $L_{\xi_1,\xi_2}$ has 
stable constant Jordan type $1[1]$, and it is an endotrivial 
module. By Lemma \ref{dim-omega}, the dimension of $L_{\xi_1,\xi_2}$
is $p^2(m+n) +1$, and hence $L_{\xi_1,\xi_2} \simeq 
\Omega^{\pm(2m+2n)}(k)$. 

Now suppose that $L_{\xi_1,\xi_2} \simeq \Omega^{-(2m+2n)}(k)$.
Then the sequence represents a non-zero 
element  $\gamma \in \HHH^{2m+2n+1}(E,k)$
which has the property that $\alpha_K^*(\gamma_K)$ is zero for any 
$\pi$-point $\xymatrix{\alpha_K:K[t]/t^p \ar[r] & KE}$. However,
because $2m+2n+1$ is both positive and odd, there is no such element.
Therefore, we must have that $L_{\xi_1,\xi_2} \simeq 
\Omega^{2m+2n}(k)$ as desired.
\end{proof}

As we show in the following proposition, there exist examples of modules 
of constant Jordan types constructed as in Proposition \ref{create1} which 
are not middle terms of extensions of endotrivial 
modules as in Proposition \ref{extend1}.  

 \begin{thm}
 \label{not-endo}
Let $E$ be an elementary abelian $p$-group of rank 3 and consider the
$kE$-module $L = L_{\zeta_1,\zeta_2,\zeta_3}$ as in (\ref{phi}), where
$\{\zeta_i\}_{i=1,2,3} \subset \HHH^2(E,k)$ 
form a system of generators of $\HHH^\bu(E,k)_{red}$. 
Then there does not exist a projective $kE$-modules $P$ such that
$M = L \oplus P$ fits in a short exact sequence of the form
\begin{equation}
\label{not}
\xymatrix{
0 \ar[r] & L^\prime \ar[r] & M \ar[r] & L^{\prime\prime} \ar[r] & 0
}
\end{equation}
with both $L^\prime, ~ L^{\prime\prime}$ endotrivial modules.
 \end{thm}

\begin{proof} Observe that $M$ has stable Jordan type $2[1]$,
so that in any short exact sequence of the form (\ref{not}) both $L^\prime,
L^{\prime\prime}$ must have stable Jordan type $1[1]$.  
We assume that such a short exact  sequence 
exists and proceed to obtain a contradiction.  
First, by eliminating projective summands we may reduce to the 
case when both $L^\prime$ and 
$L^{\prime\prime}$ are projective-free.  Since $E$ is a $p$-group,  we have 
$P \cong (kE)^t$ for some $t \geq 0$.   
Since $L^\prime, L^{\prime\prime}$ have stable constant Jordan type $1[1]$, 
they are endotrivial by Theorem~\ref{endo}.   
If $p>2$, then we immediately conclude  
$L^\prime \cong \Omega^{2n_1} (k)$, $L^{\prime\prime} 
\cong \Omega^{2n_2} (k)$
for some integers $n_1, n_2$ since   $L^\prime$, 
$L^{\prime\prime}$ have stable Jordan type $1[1]$. 
If $p=2$, then $\Dim L^\prime + \Dim L^{\prime\prime} 
= \Dim L + \Dim P \equiv 2 \,\,\,({\rm mod} \,\, \Dim kE)$.   
Since $\Dim \Omega^{n}(k) \equiv (-1)^n \,\,\, 
({\rm mod} \,\, \Dim kE)$,  we get that 
$L^\prime$, $L^{\prime\prime}$ must 
be even syzygies in the case $p=2$ as well. 
Thus, the sequence becomes
\begin{equation}
\label{seq0}
\xymatrix{
0 \ar[r] & \Omega^{2n_1} (k) \ar[r] & L \oplus (kE)^t \ar[r] &
\Omega^{2n_2} (k) \ar[r] & 0
} 
\end{equation}

Let $\{g_1, g_2, g_3\}$ be group generators of $E$, and 
consider the subgroup $F = \langle g_1, g_2 \rangle \subseteq E$.
The restriction of $\zeta_3$ to $F$ vanishes, so that 
$$
L_{|F } \ \cong \ \Omega^2(k) \oplus L_{|F,\zeta_1,\zeta_2} \oplus \proj,
$$
where $L_{|F,\zeta_1,\zeta_2}$ is constructed as in (\ref{phi}) with respect
to the group $F$.
By Lemma \ref{rk2}, 
$L_{|F,\zeta_1,\zeta_2} ~ \cong ~ \Omega^4(k)$ as a $kF$-module.
Consequently,
$$
L_{|F } \ \cong \ \Omega^{2}(k) \oplus \Omega^{4}(k) \oplus \proj.
$$

Restricting (\ref{seq0}) to $F$ and eliminating the 
projective summands at the ends, 
we obtain an exact sequence of $kF$-modules
\begin{equation}\label{seq1}
\xymatrix{
0 \ar[r] & \Omega^{2n_1}(k)  \ar[r] & 
\Omega^{2}(k) \oplus \Omega^{4}(k) \oplus \proj \ar[r] &
\Omega^{2n_2}(k)  \ar[r] & 0
}
\end{equation}

By performing a shift and eliminating excess projectives, 
we get the sequence of $kF$-modules
\begin{equation}\label{seq2}
\xymatrix{
0 \ar[r] & \Omega^{2n_1-2n_2}(k) \ar[r] & 
\Omega^{4-2n_2}(k) \oplus \Omega^{2-2n_2}(k) \oplus \proj 
\ar[r]^{\qquad \qquad \qquad \theta} &
k \ar[r] & 0
}
\end{equation}
where $\theta$ restricts to $\theta_1$ on $\Omega^{4-2n_2}(k)$ and to 
$\theta_2$ on $\Omega^{2-2n_2}(k)$.   
Since the kernel of $\theta$ has 
stable Jordan type $1[1]$, at least one of $\theta_1$, $\theta_2$ does not 
factor through a projective module when 
restricted to any $\pi$-point of $F$.   
Hence, Proposition~\ref{negcoho} implies 
that $n_2 \leq 2$.  By the same argument applied to the 
other end of Sequence (\ref{seq2}), we have that $2n_1 - 2n_2$ cannot
be less than $2-2n_2$ (i.e., $n_1 \geq 1$).  We further observe that 
if there were a non-trivial projective summand 
in the middle term of (\ref{seq2}),
then $\theta$ restricted to the projective 
summand would be surjective, and, hence, 
the kernel would have stable Jordan type different from $1[1]$. 
It follows that there is no projective summand in the middle term of 
(\ref{seq2}). Hence, we get $p^3(|4 - 2n_2|+|2-2n_2|) + 2 = 
\Dim (\Omega^{4-2n_2}(k) \oplus \Omega^{2-2n_2}(k)) 
= \Dim \Omega^{2n_1-2n_2}(k) + 1 = p^3(|2n_1 - 2n_2|) + 2$ \,
 by Lemma~\ref{dim-omega}.  Hence, 
\begin{equation}
\label{parity}
|2-n_2| + |1-n_2| = |n_1-n_2|
\end{equation}
We conclude that $n_1 +n_2 = 3$ when $n_2 \leq 1$.  
Moreover, $n_1$ and $n_2$ must have different parity.

\vspace{0.2in}
\noindent
We consider two cases: 

\begin{itemize}
\item[(I)] $n_2 \geq 0$, and 
\item[(II)] $n_2<0$.
\end{itemize} 

\vspace{0.2in}
\noindent 
{\it Case I: $n_2 \geq 0$.}   
We first show that the middle term of the sequence (\ref{seq0}) 
does not have a projective summand, that is $t=0$.

The sequence (\ref{seq0}) represents a cohomology class in 
$\widehat{\HHH}^{2n_2-2n_1+1}(E,k)$ which 
we denote by $\eta$. Restriction of (\ref{seq0}) 
to any $\pi$-point of $E$ has the form 
$$
\xymatrix{
0 \ar[r] & n[p] + 1[1] \ar[r] &(n+m)[p] + 2[1] \ar[r] &
m[p] + 1[1] \ar[r] & 0
}
$$
This sequence of $\Z/p\Z$-modules is necessarily split. 
Hence, $\eta$ vanishes upon restriction to any $\pi$-point of $E$. 

Applying $\Hom_{\stmod(kE)}(-,k)$ to the short exact sequence (\ref{seq0}), 
viewed as a distinguished triangle in $\stmod(kE)$, 
we obtain a long exact sequence  
$$\xymatrix{
\dots \ar[r] & \widehat{\Ext}^{-1}_{kE}(\Omega^{2n_1}(k),k) \ar[r]^{\delta} &
\widehat{\Ext}^{0}_{kE}(\Omega^{2n_2}(k),k) \ar[r] &
\widehat{\Ext}^{0}_{kE}(L,k)  \ar[r] & \dots \ ~.
}
$$
which is equivalent to 
\begin{equation}\label{long}
\xymatrix{
\dots \ar[r] & {\HHH}^{2n_1-1}(E,k) \ar[r]^{\cdot \eta} &
{\HHH}^{2n_2}(E,k) \ar[r] &
\widehat{\Ext}^{0}_{kE}(L,k)  \ar[r] & \dots \ ~.
}
\end{equation}
The rank of the free summand $(kE)^t$ in the middle term of 
the sequence (\ref{seq0}) equals the dimension of 
the kernel of the map 
$\xymatrix{\widehat{\Ext}^{0}_{kE}(\Omega^{2n_2}(k),k) \ar[r] &
\widehat{\Ext}^{0}_{kE}(L,k)}$.   Hence, it also equals the 
dimension of the image of the connecting homomorphism $\delta$, 
which is  multiplication by  $\eta$, a cohomology class 
of degree $\deg \eta = 2n_2-2n_1+1$.    
Therefore, to show that $t=0$  we need to show that  
multiplication by $\eta$ on ${\HHH}^{2n_1-1}(E,k)$ is trivial.

If $2n_2-2n_1+1 > 0$ then for $p>2$ the fact that $\eta$  
vanishes upon restriction to any $\pi$-point 
implies that $\eta$ is divisible by the product $\eta_1\eta_2\eta_3$  
where $\eta_1$, $\eta_2$, $\eta_3$ 
are the nilpotent generators in degree $1$ of $\HHH^*(E,k)$.  
Thus, the multiplication by $\eta$  of any odd-dimensional 
class  of positive degree is zero. 
Hence, the image of $\cdot \eta$ on $\HHH^{2n_1-1}(E,k)$ is 
trivial since $2n_1-1 \geq 1$. 
If $p=2$ then the only class which vanishes upon restriction 
to any $\pi$-point is the zero class.  
We conclude that $t=0$ in this case.    

Now assume that $2n_2-2n_1+1 < 0$. Since cohomology of $E$ 
is not periodic, multiplication by $\eta$ on $\HHH^{- \deg \eta}(E,k)$
is trivial (\cite[2.2]{BC1}).     
Since any monomial in $\HHH^{2n_1-1}(E,k)$ factors as a product of a class  of 
degree $- \deg \eta = 2n_1-2n_2 -1$ and a class  of degree $2n_2$,  
we further conclude that   
multiplication by $\eta$  on $\HHH^{2n_1-1}(E,k)$  is trivial.  
Hence, $t=0$ in this case as well.  
 
We now compare the dimensions of the terms of the 
sequence (\ref{seq0}) in which we take $t=0$. Note that $n_1$ and $n_2$
are both non-negative in the case that we are considering. 
We get
$$
6p^3+2 = \Dim L = \Dim \Omega^{2n_1}(k) + \
\Dim \Omega^{2n_2}(k) = n_1(n_1+1)p^3 +1 + n_2(n_2+1)p^3 + 1
$$
using Lemma~\ref{dim-omega}.  Hence, we get 
the following equation on $n_1,n_2$:
$$
n_1^2+n_2^2 + n_1 + n_2 = 6
$$
The only non-negative integer solutions are $(2,0)$ 
and $(0, 2)$.  This is impossible 
since $(n_1,n_2)$ must have the opposite parity by (\ref{parity}).
Hence, we obtain a contradiction as desired,  
and the case $n_2 \geq 0$ is finished. 

\vspace{0.2in}
{\it Case II.}  We now consider the case of negative $n_2$.    
The projective summand $(kE)^t$ in the 
middle term of the sequence (\ref{seq0}) 
can not be any bigger than a projective 
cover of $\Omega^{2n_2} (k)$ as otherwise, kernel of the map
$\xymatrix{(kE)^t \ar[r] &   \Omega^{2n_2} (k)}$
has a projective submodule which is then 
a direct summand of the right hand term of (\ref{seq0}).
By Lemma~\ref{proj-dim}, 
$t \leq 2n_2^2 - n_2$. 
Hence, 
\begin{align*} 
\Dim L  & \geq \Dim \Omega^{2n_2}(k) + \Dim 
\Omega^{2n_1}(k) - p^3(2n_2^2 - n_2) \\ 
& = p^3(n_2^2 - n_2) +  p^3(n_1^2+n_1) + 2 - p^3(2n_2^2 - n_2)
\end{align*}
Simplifying, and using the fact that $n_2 + n_1 = 3$, we get 
\begin{align*}
\Dim L & \geq p^3(n_1^2-n_2^2 + 3 - n_2) + 2 \\ 
& = p^3(3(n_1-n_2) + 3 - n_2) + 2 = p^3(3n_1 + 3 - 4n_2) + 2
\end{align*}
Since $n_1 \geq 1$, and $n_2 < 0$, 
we conclude that $\Dim L \geq 10 p^3 +2$.   
On the other hand, $\Dim L = 6p^3 + 2$ by 
Lemma~\ref{dim-omega} and the definition of $L$, 
and we have a contradiction. 
\end{proof}


\vspace{0.2in}

\section{Constraint on ranks}
\label{constraint}

A consequence of our constructive techniques is a new proof of a 
special case of Macaulay's Generalized Principle Ideal Theorem. 
The point is that if the coefficients are in a field of finite
characteristic, then
we can represent homogeneous elements of a multivariable polynomial
ring as elements in the cohomology ring of an elementary abelian
$p$-group. We can represent a matrix of such elements as a 
map of modules over the group algebra. Specifically, we have the 
following 

\begin{thm} \label{ranks} (See Exercise 10.9 of \cite{Eis})
Suppose that $k$ is an algebraically closed field.
Fix integers $n$, $r$ and $d_1, \dots, d_{r+1}$, with $n \geq 3$, $r \geq 2$,
and $d_j>0$ for all $j$. Let $P = k[x_1, \dots, x_n]$
be a polynomial ring in $n$ variables. Let
$A = A(x_1, \dots, x_n)$ be a $(r+1) \times r$ matrix with the 
property that every entry in column $i$ of $A$ is  
a homogeneous polynomial in $P$ of degree $d_j$ for 
all $j = 1, \dots, r+1$. Then there is
some point $\alpha \in k^n \backslash \{0\}$ such 
that $A(\alpha)$ has rank less than
$r$. Equivalently, the determinants of the $r \times r$ minors
of $A$ (which are elements of $P$) have a common non-trivial zero.
\end{thm}

\begin{proof}
Assume first that the characteristic of $k$ is $p > 0$, as in the rest
of the paper. Let $E$ denote an elementary abelian $p$-group of 
rank $n$.  As recalled in (\ref{orderp}),  $\HHH^*(E,k)$ contains a 
polynomial subring $Q \cong k[\zeta_1, \dots, \zeta_n]$, where the 
elements $\zeta_i$ are in degree 2 if $p >2$ and in degree 1 if $p=2$.
For the purposes of the argument, we assume that $p > 2$. The proof 
in the even characteristic case is very similar. 

Let $A = (a_{i,j})$ where for each $i$ and $j$, $a_{i,j} = a_{i,j}
(x_1, \dots, x_n)$ is a homogeneous polynomial of degree $d_j$. Then,
$a_{i,j}(\zeta_1, \dots, \zeta_n)$ is an element in $\HHH^{2d_j}(E,k)$.
Moreover, such an element is uniquely represented by a cocycle 
$$
\xymatrix{
a'_{i,j}(\zeta_1, \dots, \zeta_n): \ \ \Omega^{2d_j}(k) \ar[r] & k.
}
$$
Now we let $A'$ be the map 
$$
\xymatrix{
A': \ \ \bigoplus\limits_{t = 1}^{r+1} \Omega^{2d_t}(k) \ar[r] & k^r
}
$$
whose matrix is $A' = (a'_{i,j}(\zeta_1, \dots, \zeta_n))$.  

We proceed to prove the theorem by contradiction, observing
that if $A(\alpha)$ had rank $r$ for all $\alpha \in k^n \backslash \{0\}$,
then, because $k$ is algebraically closed (so that the $k$-rational points of $\Pi(E)$
are dense), the kernel $L$ of $A^\prime$ would be a module of constant
Jordan type (with stable Jordan type $1[1]$) as in Theorem  \ref{create1}.  
By Theorem \ref{endo}, $L$  is an endotrivial module. 
Hence, $L \cong  \Omega^{2m}(k) + \proj$ for some $m$.  Since  
$\bigoplus\limits_{t = 1}^{r+1} \Omega^{2d_t}(k)$ does 
not have projective summands, 
we conclude that $L \cong  \Omega^{2m}(k)$. 

 As in the proof of Theorem \ref{not-endo},
we can use a dimension argument to ascertain the value of $m$.  
Let $H \subset E$ be an elementary abelian subgroup of rank $2$. 
Restricting to $H$ and 
eliminating the projective summand in $L$ and in the domain of $A'$,  
we get that the dimension of the projective-free 
part of $L\downarrow_H$ is precisely
$p^2\sum_{i=1}^{r+1} d_i +1$ by Lemma \ref{dim-omega}(4). 
Consequently, by the same 
lemma $m = \sum_{i=1}^{r+1} d_i$.

Now let $H^\prime \subset E$ be an elementary abelian 
$p$-subgroup of rank 3, and let $L^\prime$ 
be the projective-free part of the restriction of $L$ to $H^\prime$. 
Since $L^\prime \simeq \Omega^{2m}(k)$ as 
$H^\prime$-modules, Lemma~\ref{dim-omega}(5) implies  that 
\begin{equation}
\label{dim1}
\Dim L^\prime =  p^3m(m+1) +1 =  
p^3(\sum_{i=1}^{r+1} d_i)(\sum_{i=1}^{r+1} d_i +1) + 1
\end{equation}
On the other hand, $L^\prime$ is the projective-free 
part of the kernel of the map $A'$ restricted to $H^\prime$. 
Applying Lemma~\ref{dim-omega}(4) 
to compute the dimension of the $H^\prime$-module 
$\bigoplus\limits_{t = 1}^{r+1} \Omega^{2d_i}(k)$
we get
\begin{equation}
\label{dim2}
\Dim L^\prime = p^3 \sum_{t =1}^{r+1} d_i(d_i+1) +1. 
\end{equation}
As all $d_i>0$, the formula (\ref{dim1}) clearly 
yields a greater value than (\ref{dim2}).  
Thus, we get a contradiction. 

\vspace{0.1in}

Now, we consider a field of characteristic 0, still denoted $k$.
Let $R \subset k$ be the ring finitely generated over $\bZ$ by the 
coefficients of the (homogeneous polynomial) entries of $A$.  
The $r+1$ determinants of the $r\times r$ minors of $A$  define a closed 
subscheme $Z$ of the projective scheme $\bP_R^{n-1}$.  By the 
preceding argument for fields of positive characteristic, $Z$ intersects 
each geometric fiber of 
$\xymatrix@-.5pc{\bP^{n-1}_R \ar[r] & \Spec R}$ above a point 
of $\Spec R$ with positive residue characteristic.

Observe that the geometric points of $\Spec R$ whose residue characteristics
are positive are dense in $\Spec R$.  
Hence, the image of $Z\subset \bP^{n-1}_R$
in $\Spec R$ is both closed and dense and thus all of $\Spec R$.
We conclude that $Z$ intersects every geometric fiber of 
$\xymatrix@-.5pc{\bP^{n-1}_R \ar[r] & \Spec R}$,
including that given by $\xymatrix@-.5pc{R \ar[r] & k}$. 
\end{proof}


\vspace{0.2in}

\section{Auslander Reiten Components}
\label{AR}
Our objective in this section is to show how the Auslander-Reiten 
theory of almost split sequences can be used to construct modules
having certain specific constant Jordan types. 
One of our main results is that if one module in a component of 
the Auslander-Reiten quiver of $kG$ has constant Jordan type
and if $k$ is perfect, then every module in the component 
has constant Jordan type. 

We begin this section by briefly recalling a few basic facts
concerning the Auslander-Reiten theory of almost split sequences, 
sometimes called Auslander-Reiten sequences. We refer the reader 
to \cite{ARS} or \cite[I.4]{B} for more detailed accounts. 

Suppose that $A$ is a finite dimensional algebra over the field $k$. 
A sequence of A-modules
$$
\xymatrix{ 
0 \ar[r] & L \ar[r]^{\zeta} & M \ar[r]^{\gamma} & N \ar[r] & 0 
}
$$ 
is an almost split sequence if $M$ and $N$ are indecomposable, the 
sequence is not split, and one of the following two equivalent conditions
holds:
\begin{enumerate}
\item
Any $A$-module map $\,\sigma: X \longrightarrow N$ 
which is not a splittable epimorphism, 
admits a factorization  $\sigma = \gamma\theta$
for some homomorphism  $\,\theta: X \longrightarrow M$. 

\item 
Any $A$-module map \, $\sigma: L \longrightarrow Y$
which is not a splittable monomorphism, 
admits a factorization $\sigma = \theta\zeta$
for some homomorphism \, $\theta: M \longrightarrow Y$.
\end{enumerate}

\noindent 
(see \cite[V.1]{ARS}).
If $M$ is an indecomposable non-projective 
$A$-module, then the almost split sequence of $A$-modules 
ending in $M$  is unique up to isomorphism and 
has the form
$$\xymatrix{0 \ar[r]& \tau M \ar[r]& B \ar[r]& M \ar[r]& 0}$$
where 
$\tau = \D \circ \Tr$ is the {\it translation} operator  on the set 
of isomorphism classes of  indecomposable non-projective $A$-modules  
defined as the composition of the transpose and 
duality functors (see \cite[VII.1]{ARS}).  Similarly,  there is  a 
unique up to isomorphism almost split sequence starting with an indecomposable 
non-injective module $L$.
If $A$ is a symmetric algebra (in particular, 
if $A = kG$ for a finite group $G$), then $\tau = \Omega^2$.   
We give description of $\tau$  in the case of any 
finite group scheme $G$ in the  proof of Lemma~\ref{tau}.

Let $L$, $M$ be indecomposable $A$-modules. 
A map $\xymatrix{\gamma: L \ar[r] & M}$ is an {\it {irreducible}} map
if the following two conditions hold: 
\begin{enumerate}
\item[(a)] $\gamma$ is neither a splittable monomorphism 
nor a splittable epimorphism
\item[(b)] for any factorization $\gamma = \mu\nu$ either $\mu$ is a 
splittable monomorphism (has a right inverse) or $\nu$ is a splittable
epimorphism (has a left inverse) 
\end{enumerate}
(see \cite[V.5]{ARS}). 
Irreducible maps are closely related to almost split 
sequences as the following lemma   indicates.
\begin{lemma}
\label{irred} (\cite[V.5.3]{ARS})  Let
$ 
0 {\longrightarrow}  L \stackrel{\zeta}{\longrightarrow}  
M \stackrel{\gamma}{\longrightarrow} N  \longrightarrow 0 
$
be an almost split sequence, and let $M = 
\bigoplus\limits_{i=1}^r M_i$ where each
$M_i$ is an indecomposable module.  Then the maps $\xymatrix{\zeta_i:
L \ar[r] & M_i}$ and $\xymatrix{\gamma_i: M_i \ar[r] & N}$ that make
up $\zeta$ and $\gamma$ are irreducible maps. Moreover, every irreducible
map occurs in such a way in some almost split sequence.

\end{lemma}

The next result shows that irreducible maps 
originating in projective modules, and consequently, 
almost split sequences with projective summands 
in the middle term  have a very special form.  We now 
take the algebra $A$ to be the group algebra $kG$ for 
a finite group scheme $G$.   Since $kG$ is a Frobenius algebra, 
projective modules are injective  and vice versa.  

\begin{lemma} \label{proj-in-ass} (\cite[V.5.5]{ARS})
Let $G$ be a finite group scheme, and $P$ be an indecomposable
projective $kG$-module. Then, up to scalar multiple, the only irreducible
map originating at $P$ is the quotient map $\xymatrix@-.6pc{P \ar[r] &
P/\soc(P)}$. Consequently, the only almost split sequence in which $P$
occurs is the almost split sequence ending in $P/\soc(P)$.  
This sequence  has the form 
\begin{equation}
\label{ass-kg}
\xymatrix{0 \ar[r]& \Rad(P) \ar[r] &  P \oplus 
\Rad(P)/\soc(P) \ar[r] & \Rad(P)/\soc(P) \ar[r] & 0  } 
\end{equation} 
\end{lemma}

For a finite group scheme $G$, the {\it stable Auslander-Reiten quiver} 
of $G$, denoted $\Gamma^s_G$, 
is a valued quiver whose vertices  are  isomorphism classes of 
indecomposable non-projective $kG$-modules, 
and whose arrows are induced by irreducible morphisms 
(see \cite[VII.1,4]{ARS}). Namely, if there is an irreducible 
map $M \to N$ for non-projective indecomposable modules $M$,$N$, 
then there is an arrow $\xymatrix{[M] \ar[r]& [N]}$ in 
the Auslander-Reiten quiver of $G$.  By Lemma~\ref{irred}, there 
is an arrow $\xymatrix{[M] \ar[r]& [N]}$ if and only if there exists 
an  almost split sequence  $0 \longrightarrow \tau N \longrightarrow 
\widetilde M \longrightarrow N \longrightarrow 0$  
such that $M$ is an indecomposable direct summand of $\widetilde M$. 
A (stable) Auslander-Reiten component of an indecomposable non-projective 
module $M$ is a connected component 
of the Auslander-Reiten quiver $\Gamma^s_G$ containing the vertex $[M]$.  
In addition, the Auslander-Reiten 
quiver is equipped with the translation automorphism 
$\xymatrix{\tau: \Gamma^s_G \ar[r]& \Gamma^s_G}$ induced by the 
translation operator $\tau$.

An Auslander-Reiten component $\Theta$ of $\Gamma^s_G$ has the form
\begin{equation}
\label{tree}
\Theta \simeq \Z[\vec{\Delta}_\Theta]/H
\end{equation}
where $\vec{\Delta}_\Theta$ is a directed tree, and 
$H \subset \Aut(\Z[\vec{\Delta}_\Theta])$ 
is a group acting ``admissibly"  on $\vec{\Delta}_\Theta$ 
(see \cite[VII.4]{ARS} for definitions and \cite[I.4.15.6]{B} 
for the structure isomorphism).
The underlying non-directed tree $\Delta_\Theta$ is 
uniquely determined by the component 
$\Theta$ and $\Theta$ is said to have {\it tree class}  or {\it type} of 
$\Delta_\Theta$.

Later in the section we use the following  result, which is due to Erdmann 
\cite{Er} in the case of finite groups. For arbitrary finite group
schemes, there are similar results of  Farnsteiner \cite{F1}, 
\cite{F2}.   

\begin{thm} \label{erd-farn} \cite{Er}
Suppose that $G$ is a finite group and that $M$ is a $kG$-module
which lies in a block of wild representation type. Then the 
Auslander-Reiten component of $M$ has tree class $A_{\infty}$. 
\end{thm}

The tree class $A_{\infty}$ tells us a lot about the structure of the almost
split sequences.  The vertices of a stable Auslander-Reiten component 
of such a tree class 
are  obtained by  iterated applications of the translation operator $\tau$  
to a distinguished collections of vertices    which correspond to 
isomorphism classes of non-projective indecomposable  modules 
$X_0, X_1, X_2, \dots$ such that $X_0$ is
the beginning vertex and for each $i \geq 0$ there is an irreducible map  
$X_i \to X_{i+1}$. That is, we have a tree 
$$\xymatrix{
\Delta:  \quad X_0  \ar[r]&  X_1 \ar[r] & \ldots  \ar[r] 
& X_i \ar[r] & X_{i+1} 
\ar[r] & \ldots }
$$
such that the component is isomorphic to $\Z[\vec{\Delta}]/H$ 
as in (\ref{tree}). 
An almost split sequence involving 
modules in the component has the form
$$
\xymatrix{ 
0 \ar[r] & \tau^{n+1}X_i \ar[r] & \tau^{n+1}X_{i+1} \oplus 
\tau^{n} X_{i-1} \ar[r] & \tau^{n} X_i \ar[r] & 0 
}
$$
for some $n$ as long as $i \geq 1$. Here, $\tau^0$ is the identity
operator, $\tau^{-1} = Tr \circ D$, {\it etc}. On the 
boundary of the component 
($i = 0$),  the almost split sequence has the form
$$
\xymatrix{
0 \ar[r] & \tau^{n+1} X_0 \ar[r] & \tau^{n+1} X_{1} \oplus P 
\ar[r] & \tau^{n} X_0 \ar[r] & 0 
}
$$
where $P$ might be zero or might be an indecomposable projective 
module (see \cite[VII.4]{ARS}). 
However, as stated in Lemma~\ref{proj-in-ass},  
each indecomposable projective module
occurs in a unique almost split sequence (which is not necessarily in the 
component under consideration).

Before showing how almost split sequences can be used to
generate indecomposable modules of constant Jordan type,
we need the following technical lemma.
We fix some notation which differs from
standard notation in the case of finite groups:  
In what follows, let $A \subset B$ be rings, and $M$ 
be a left $A$-module. Then
$$
\Coind_A^B M = B \otimes_A M
$$ 
   
Let $G$ be a finite group scheme, and
$\xymatrix@-.5pc{\alpha_K: K[t]/t^p \ar[r] & KG}$ 
be a $\pi$-point.   We denote by $K\langle \alpha_K(t) \rangle$ 
the subalgebra of $KG = KG_K$ generated by $\alpha_K(t)$.

\begin{lemma}
\label{periodic}
Let $G$ be a finite group scheme and let $\alpha_K: K[t]/t^p \to KG$ be
a $\pi$-point of $G$.  Let $N$ be a finite dimensional 
$K[t]/t^p$-module which is not projective 
and set $M = \Coind_{K\langle \alpha_K(t) \rangle}^{KG} N$.  Then
$$\Pi(G_K)_M  ~ = ~ \{ [\alpha_K] \} ~ \subset ~ \Pi(G_K).$$
\end{lemma}

\begin{proof}
If $\Pi(G_K)_M$ were empty, then $M$ would be 
projective, and, hence, injective.   
On the other hand, the Eckmann-Shapiro Lemma (\cite[I.2.8.4]{B})
would enable us to then conclude that 
$\Ext_{G_K}^{*>0}(M,K) = \Ext_{K\langle \alpha_K(t) \rangle}^{*>0}(N,K) = 0$ 
which would contradict our assumption that $N$ is not projective.  Thus,
$\Pi(G_K)_M $ is non-empty.   

Let $U_K \subset G_K$ be a unipotent abelian subgroup scheme through
which $\alpha_K$ factors.  The proof of \cite[4.12]{FPS} which is stated
for induction rather than coinduction implies that 
$$\Pi(G_K)_M~  \subset ~
\rm{im}\{ \Pi(U_K) \to \Pi(G_K)\}$$
so that we may assume that $G_K = U_K$ is a unipotent abelian finite
group scheme over $K$.

After possibly replacing $K$ by some purely inseparable extension which
does not change the space $\Pi(G_K)_M$, we may by  \cite[14.4]{W} assume 
that $KG_K$ is isomorphic (as an algebra) to 
$K[T_1, \dots, T_n]/(T_1^{p^{e_1}},\ldots,T_n^{p^{e_n}})$
for suitable choice of $n, e_1,\ldots,e_n$.   
Let $t_i = T_i^{p^{e_i-1}}$, and recall that any 
$\pi$-point $\beta_L: L[t]/t^p \to 
LG_K$ must send $t$ to a sum of monomials in 
$T_1,\ldots, T_n$ at least one of which is a 
non-linear scalar multiple of some $t_i$ and each of which are divisible by 
some (possibly varying) $t_i$.   By a change of generators, we may arrange that
$\alpha_K(t) = t_1 + p(T)$ where each monomial of the polynomial $p(T)$ is
a non-scalar multiple of some $t_i$.   

In order to verify that $[\beta_L] \notin \Pi(G_K)_M$
for $[\beta_L] \not= [\alpha_K]$, we may choose a representative $\beta_L$ of
$[\beta_L]$ which is linear in the $t_i$'s.   
Assuming $[\beta_L] \not= [\alpha_K]$,
we may change generators once again so that $\alpha_K(t)$ retains the form 
$t_1 + p(T)$ as above and $\beta_L(t) = t_2$.  
The condition that $[\beta_L] \notin \Pi(G_K)_M$ is equivalent to the
condition that $\beta_L^*(M_L)$ is free.  Clearly, it suffices to assume
that $N$ is indecomposable of the form $K[t]/t^i, i < p$.
Then $M_L ~ \cong ~ L[T_1, \dots, T_n]/(t_1^p, \dots, t_n^p, (t_1+p(T))^i)$
which is free over $L{\langle t_2 \rangle}$ with a monomial basis 
$\{T_1^{j_1}, T_2^{j_2}, \dots, T_n^{j_n}\}$, where 
$ 0 \leq j_1 < ip^{e_1-1}, 0\leq j_2 < p^{e_2-1}, 0\leq j_3 < p^{e_3} 
\dots 0\leq j_n < p^{e_n}$.
\end{proof}

We now prove a local restriction result for almost split exact 
sequences of $kG$-modules.

\begin{prop} 
\label{restr}
Let $G$ be a finite group scheme  such 
that the dimension of $\Pi(G)$ is at least $1$, 
and let $M$ be an indecomposable non-projective
$kG$-module of constant Jordan type.  Assume 
that one of the following conditions hold: 
either $M$ is absolutely indecomposable or $k$ is perfect. 
Consider the almost split sequence of $kG$-modules 
$$
\xymatrix{
\CE: 0 \ar[r] & N \ar[r] & B \ar[r] &  M \ar[r] & 0.
}
$$ 
Then for any $\pi$-point  $\xymatrix@-.6pc{\alpha_K: K[t]/t^p \ar[r] & KG}$,
$\alpha_K^*(\CE_K)$
is a split short exact sequence of $K[t]/t^p$-modules.
\end{prop}

\begin{proof}
If $M_K$ is decomposable, write  $M_K \cong \oplus M_K^i$ 
as a direct sum of indecomposable 
$KG_K$-modules. Since  in this case $k$ is perfect, 
Theorem \cite[3.8]{Kas} implies that the almost split 
sequence $\CE_K$ is a direct sum of
almost split sequences $$
\xymatrix{
\CE_K^i: 0 \ar[r] & N_K^i  \ar[r] & B_K^i \ar[r] &  M_K^i \ar[r] & 0.
}
$$ 
Thus, it suffices to prove that 
$\alpha_K^*(\CE_K^i)$ is split for each $i$.   Hence, we may assume 
that $M_K$ is neither projective nor decomposable, and that
$$
\xymatrix{
\CE_K: 0 \ar[r] & N_K  \ar[r] & B_K \ar[r] &  M_K \ar[r] & 0
}
$$ 
is an almost split sequence of $KG_K$-modules. 

Let $\widetilde M_K=\Coind_{K\langle 
\alpha_K(t) \rangle}^{KG}(\alpha_K^*(M_K))$.
We have a commutative diagram
\begin{equation}
\label{dia}
\begin{xy}*!C\xybox{%
\xymatrix{ \Hom_{KG_K}(\widetilde M_K, B_K)  \ar[d]^{\cong}  \ar[r]  & 
\Hom_{KG_K}(\widetilde M_K, M_K)
\ar[d]^{\cong}\\
\Hom_{K\langle \alpha_K(t) \rangle}(\alpha_K^*(M_K), \alpha_K^*(B_K))  \ar[r]&
\Hom_{K\langle \alpha_K(t) \rangle}(\alpha_K^*(M_K),\alpha_K^*(M_K))   }}
\end{xy}
\end{equation}
where the vertical arrows are isomorphisms by the 
Eckmann-Shapiro Lemma.
If $\alpha_K^*(\CE_K)$ were not split, then the lower horizontal arrow 
of (\ref{dia})  and thus also the upper horizontal arrow of (\ref{dia})
would not be surjective.  On the other hand, the defining property of
almost split sequences  would then imply
that $\widetilde M_K$  must have $M_K$ as a direct summand.  If so, then
$\Pi(G_K)_{M_K} \subset \Pi(G_K)_{\widetilde M_K}$. Since $M_K$ is a 
non-projective module of constant Jordan type, 
we have $\Pi(G_K)_{M_K} = \Pi(G_K)$, and, hence, the support of 
$M_K$ has dimension at least $1$. 
Since  $\Pi(G_K)_{\widetilde M_K}$ consists of 
only $1$ point by Lemma \ref{periodic}, 
we obtain a contradiction.
\end{proof}

\begin{lemma}  
\label{tau}
Let $G$ be a finite group scheme, and $M$ be an indecomposable non-projective 
finite-dimensional $kG$-module.   Then $M$  is a module of 
constant Jordan type if and only if 
$\tau M$ is a module of constant Jordan type. 
Moreover, if $M$ is a module of constant Jordan type 
then the stable Jordan types of $M$ and $\tau M$ are the same. 
\end{lemma}

\begin{proof}
Since $kG$ is self-injective, the translation 
operator $\tau$ is isomorphic 
to $\mathcal N \circ \Omega^2$ where $\mathcal N$ 
is the Nakayama functor  defined as 
$$\mathcal N = \Hom_{kG}(-,kG)^\#$$ 
(see \cite[IV.3]{ARS}).

The Heller shift $\Omega^2$ preserves the property of 
being of constant Jordan type, and, 
moreover, preserves the stable Jordan type  of a module 
of constant Jordan type (\ref{heller}). 
Hence, to show that $\tau$ preserves modules of constant 
Jordan type, we need to demonstrate this for 
the Nakayama functor. 

By \cite[I.8.12]{J} we can exhibit a character $\delta$ of $G$ with
the property that  
$kG^\# \otimes k_\delta \simeq kG$ 
is a two-sided $kG$-isomorphism, where $k_\delta$ 
is the one-dimensional $kG$-representation with 
$kG$ acting trivially on the right and by the 
character $\delta$ on the left. Hence, we get the 
following isomorphisms of $kG$-modules
$$\Hom_{kG}(M,kG) \simeq \Hom_{kG}(M, kG^\# \otimes k_\delta) \simeq 
\Hom_{kG}(M \otimes k_\delta^\#, kG^\#) \simeq $$
$$\Hom_{kG}(M \otimes k_\delta^\#, \Hom_k(kG,k)) \simeq 
\Hom_k(M \otimes k_\delta^\#, k) = M^\# \otimes k_\delta$$
Hence,
$$\mathcal N(M) = \Hom_{kG}(M,kG)^\# \simeq 
(M^\# \otimes k_\delta)^\# \simeq M \otimes k_\delta^\#$$
as $kG$-modules.  
Let $M$ be a module of constant Jordan type. 
Since the module $k_\delta$ is 1-dimensional, 
it is of trivial constant Jordan type.
Hence, the module $\mathcal N(M) \simeq M \otimes k_\delta^\#$ 
is also of constant 
Jordan type by Corollary \ref{tensor}.  Moreover, the 
Jordan type of $M \otimes k_\delta^\#$ 
is the same as that of $M$.

To show that $\tau M$ being of constant Jordan
type implies that $M$ is of constant Jordan type, 
we merely observe that the operator $\tau$  has an inverse 
given by $\Omega^{-2} \circ \mathcal N^{-1}$. 
   Hence, if $\tau M$ is of
constant Jordan type then so is $M$. 
\end{proof}

The following theorem asserts that whether or not an indecomposable 
$kG$-module $M$ has constant Jordan type is a function of the
connected components of the stable Auslander-Reiten quiver of $kG$.

\begin{thm}
\label{component}
Let $G$ be a finite group scheme, and let $M$ be an 
indecomposable non-projective module of constant Jordan type.
Let $\Theta$ be a component of the stable Auslander-Reiten 
quiver of $kG$ containing  the module $M$.  Assume 
that one of the following 
conditions hold: either all vertices of $\Theta$   
are absolutely indecomposable, or $k$ is perfect. 
Then for any $[N] \in \Theta$, the module  $N$ 
has constant Jordan type.

\end{thm}
\begin{proof}
We first consider the case when $\Dim \Pi(G) = 0$. By 
Theorem \ref{connected}, $\Pi(G)$ consists of one point.  Hence, 
any module is tautologically a module of constant Jordan type. 

We may therefore assume that $\Dim \Pi(G) \geq 1$.  
Let $[N]$ be any predecessor 
of $[M]$ in the stable quiver component $\Theta$.  
Then there exists an almost split sequence
$$
\xymatrix{
\CE: 0 \ar[r] & \tau M \ar[r] & B \ar[r] &  M \ar[r] & 0
}
$$ 
 such that $N$ is a direct summand of $B$.    
By Lemma \ref{tau}, $\tau M$ is a module of constant Jordan type.
Let $\alpha_K: K[t]/t^p \to KG_K$ be a $\pi$-point.   
By Proposition \ref{restr}, $\alpha_K^*(\CE_K)$ splits.  
Thus, $\alpha_K^*(B_K) = \alpha^*(\tau M_K) \oplus \alpha^*(M_K)$.  
We conclude that $B$ is a module of constant Jordan type, so that 
Theorem \ref{summand} implies that $N$ has constant Jordan type. 

Now let $[N]$ be any successor  of $[M]$, i.e., 
there is an arrow $\xymatrix{[M] \ar[r] & [N]}$.    
By \cite[V.1.12]{ARS} and \cite[V.5.3]{ARS}, 
there is an arrow $\xymatrix{[\tau N] \ar[r] & [M]}$.  
Applying the argument above to  $\tau N$ and $M$, 
we conclude that $\tau N$ has constant Jordan type.   By Lemma~\ref{tau}, 
$N$ has constant Jordan type.

Since $\Theta$ is connected, the argument is finished by induction.    
\end{proof}

To prove the following ``realization of constant types" 
result, we appeal to the work of 
K. Erdmann \cite{Er} in the case of finite groups and 
that of R. Farnsteiner \cite{F1}, \cite{F2} 
for arbitrary finite
group schemes (see \ref{erd-farn}).   Namely, a result of \cite{EH} 
(see also \cite{F2}) following earlier work of 
Webb \cite{Webb} asserts that if $kG$ has wild representation type then the
Auslander-Reiten component of the trivial module has a very restricted
form.  Results of Erdmann and Farnsteiner 
assert that under hypotheses specified in the 
theorem below, the Auslander-Reiten
component of the trivial module must have type  $A_{\infty}$.

\begin{thm} 
\label{stableone}  
Let $G$ be a finite group scheme (over $k$ algebraically closed)
 satisfying one of the following
conditions: either   $G$ is a finite group which has $p$-rank at 
least 2 and whose Sylow $p$-subgroup is not dihedral or semi-dihedral
or $\Pi(G)$ has dimension at least 2.  Then 
for any $n$ there exists an indecomposable 
module of stable constant Jordan type $n[1]$.
\end{thm}

\begin{proof}
By \cite[2]{Er} in the case of finite groups and 
\cite[3.3]{F2} for arbitrary finite
group schemes, our  assumptions imply that Auslander-Reiten quiver
of the  component containing  the 
trivial module must have type $A_{\infty}$. 

Let $V_n$ be an indecomposable module representing the 
$n^{\rm th}$ vertex of the tree underlying the component  
containing the trivial module.  
The bottom vertex has label 
$0$. By Proposition \ref{component}, 
$V_n$ has  constant  Jordan type 
for every $n$. Let ${\ul a}_n$  be the stable Jordan type of $V_n$.  
Proposition \ref{restr} and Lemma \ref{tau}  imply that the middle term $B_n$
of the  almost split sequence 
$$
\xymatrix{
0 \ar[r] & \tau V_n \ar[r] & B_n \ar[r] &  V_n \ar[r] & 0
}
$$ 
has stable constant Jordan type  $2{\ul a}_n$.

Since the tree class is $A_\infty$,  $V_1$ must be 
the only non-projective indecomposable 
summand of the middle term of the almost split sequence for $V_0$:
$$
\xymatrix{
0 \ar[r] & \tau V_0 \ar[r] & B_0 \ar[r] &  V_0 \ar[r] & 0
}
$$ 
Hence, ${\ul a}_1 = 2{\ul a}_0$.
The middle term of the almost split sequence for $V_1$
$$
\xymatrix{
0 \ar[r] &  \tau V_1 \ar[r] & B_1 \ar[r] &  V_1 \ar[r] & 0
}
$$ 
has two indecomposable non-projective summands, one of 
which is isomorphic to $\tau V_0$. Hence,
the other summand has stable constant Jordan type 
$2{\ul a}_1 - {\ul a}_0 = 3{\ul a}_0$.
Proceeding by induction, we see that the module 
which represents the $n^{th}$ node in this $A_{\infty}$ tree
has stable constant Jordan type ${\ul a}_n = n {\ul a}_0$.

We immediately conclude that $k$ must be at the 
bottom node since the stable Jordan type of $k$ is $1[1]$.   
Hence, ${\ul a}_0 = 1[1]$.   Therefore,  $V_n$ is 
an indecomposable module of stable constant Jordan type $n[1]$. 
 \end{proof}


\vspace{0.2in}

\section{Questions and Conjectures}
\label{qc}

	We offer a few broad questions as well as specific conjectures
which provide challenges for further investigation.

\begin{ques}
\label{q1}
For a given finite group scheme $G$, what Jordan types are realized as the
Jordan type of finite dimensional $kG$-modules with constant Jordan
type?
\end{ques}

Certainly, there are constraints as the following examples illustrate.

\begin{ex}
Let $G$ be a quasi-elementary abelian group scheme, $G = \Gas \times E$
with $E$ and elementary abelian $p$-group of rank $r$ and $p > 2$.  Then any 
$kG$-module $M$ of constant Jordan type of stable type $1[1]$  is endotrivial, 
and hence of the form $\Omega^i(k)$ \cite{Da}.  The Jordan type of such a 
module has the form $1[1] + mp^{r+s-1}[p]$ for some $m \geq 0$ (see Lemma
\ref{dim-omega}).  
\end{ex} 

\begin{ex}
We verify that there does not exist a finite dimensional $kE$-module of 
constant Jordan type $[2] + [p]$ for $p > 3$ for $E$ an elementary
abelian $p$-group of rank 2, which the reader can see as
very limited evidence for Conjecture \ref{c2}.
Suppose $V$ is such a $kE$-module and write
 $kE = k[x,y]/(x^p,y^p)$.  Consider the $k$-vector space basis 
$u, xu, v, xv, x^2v, \ldots, x^{p-1}v$ for $V$.

	We will show that some linear combination $y - bx$ 
satisfies $(y-bx)^{p-1}V = 0$, so that the Jordan form associated
to $y-bx$  has no block of size $p$.  Observe that $y(v)$ written 
in our given basis
has coefficient 0 for $v$ because $y$ is nilpotent and only $v$
in our basis satisfies $x^{p-1}v \not= 0$.  Second, suppose that 
$y(v)$ has coefficient $b$ for $xv$ and consider $y-bx$.  Then 
once again $(y-bx)(v)$ has coefficient 0 for $v$ and by construction 
coefficient 0 for $xv$.  Let us replace $y$ by $y-bx$, so that
$y(v) \in \text{span}\{ u,xu,x^2v, \ldots, x^{p-1}v \}$. 

	If we apply $y$ to the given basis, the only basis element
with the property that $y$ applied to it can have non-zero coefficient
for $u$ is $u$ itself, since any of the other basis elements would 
have to be annihilated by $x^2$ and thus in the image of $x$.  Since
$y$ is nilpotent, we conclude that 
$y(u) \in \text{span}\{ xu,x^2v, \ldots, x^{p-1}v \}$ 
so that
$y^2(u) \in \text{span}\{ x^4v, \ldots, x^{p-1}v \}$.

	Thus, $y^3u$ and $y^3v$ are both contained in 
$x^4V$, and we conclude that $y^{p-1}(x^iv) = 0 = y^{p-1}(x^iu)$.
\end{ex}

The challenge of Question \ref{q1} seems more interesting 
if we work stably, so that
we identify two Jordan types $\ul a = a_p[p] + \cdots + a_1[1], \
\ul b = b_p[p] + \cdots  + b_1[1]$  provided that $a_i = b_i, ~ i\not= p$.

\begin{ques}
For which finite group schemes $G$ is every stable Jordan type the
Jordan type of a finite dimensional $kG$-module of constant
Jordan type?
\end{ques}

The following is a specific conjecture would be a step towards
answering the previous question.  

\begin{conj}
\label{c2}
Let $E$ be an elementary abelian $p$-group of rank $\geq 2$, with $p > 3$.
Then there does not exist a finite dimensional $kE$-module of stable constant
Jordan type $[2]$.  
\end{conj}

Andrei Suslin has formulated the following 
intriguing question whose affirmative
answer would in particular verify the preceding conjecture.

\begin{ques}
\label{ques2}
Let $E$ be an elementary abelian $p$-group of rank 2, with $p > 3$.
Let $M$ be a $kE$-module with constant Jordan type $\sum_i a_i[i]$
and let $i$ be an integer, $1 < i < p$.  Is it the case that if 
$a_i \not= 0 $, then either $a_{i+1} \not= 0$ or $a_{i-1} \not= 0$?
\end{ques}

Of course, if Conjecture \ref{c2} is valid, then it 
follows that there is no module
of stable constant Jordan type $[2]$ for any finite group scheme $G$
containing a quasi-elementary subgroup scheme $H = \bG_{a(r)} \times E$
such that the rank $s$ of $E$ plus $r$ is greater or equal to $2$.

We can make many other ``non-existence conjectures" such as the
following.  We recall (see Example \ref{elemex}) that for $E$  an elementary abelian $p$-group of rank 
$n \geq 2$ there exists a $kE$-module whose stable type is constant of 
type $1[2] + (n-1)[1]$

\begin{conj}
Let $E$ be an elementary abelian $p$-group of rank $n \geq 2$, with $p > 3$.
There does not exist a $kE$-module whose 
stable type is constant of type $1[2] + j[1]$ with $j \leq n-2$.
\end{conj}

We next formulate questions of a more qualitative nature.

\begin{defn}
Let $G$ be a finite group scheme over a field $k$ (of characteristic $p > 0$).
We denote by $\cL~  = ~ \bN^p$ the (additive) 
lattice of Jordan types over $k$.  
We denote by 
$$\R(G) ~ \subset ~ \cL$$
the sublattice of those Jordan types which can be realized as the 
Jordan types
of $kG$-modules of constant Jordan type.
\end{defn}

\begin{ques}
For which finite group schemes $G$ is $\cL/\R(G)$ finite?  
Among such finite group
schemes, how does the invariant $\cL/\R(G)$ behave?

For those finite group schemes $G$ for which $\cL/\R(G)$ is infinite, 
can we give
some interpretation of the rank of this quotient in more familiar terms?
\end{ques}

In view of our discussion involving Auslander-Reiten almost split 
sequences, it seems
of considerable interest to consider $\I(G)$ as defined below.

\begin{defn}
Let $G$ be a finite group scheme over a field $k$ (of characteristic $p > 0$) 
and 
let $\ol \R(G) ~ \subset ~ \ol \cL ~ = ~ \bN^{p-1}$ 
denote the subset of those stable Jordan
types realizable as the stable Jordan type of a finite 
dimensional $kG$-module
of constant Jordan type.  Further, let us 
denote by $\I(G) ~ \subset ~ \ol \R(G) ~ \subset ~ \ol \cL ~ = ~ \bN^{p-1}$ 
the subset of 
stable Jordan types 
which are the Jordan types of {\it indecomposable} $kG$-modules
of constant Jordan type.
\end{defn}

\begin{ques}
For which $G$ is $\I(G)$ closed under addition?
\end{ques}

\begin{remark}
If $G$ is the Klein four group, $G \simeq \Z/2\Z \times \Z/2\Z$, then the 
only non-projective indecomposable modules
of constant Jordan type are Heller shifts of the trivial module $k$.  Thus, 
for this choice of $G$, $\I(G)$ is not closed under addition.
\end{remark}

One is tempted to ask many questions concerning how the realizability
of modules of constant type behaves with respect to change of finite
group scheme.  We ask one such question.

\begin{ques}
For which $H \subset G$ does restriction induce a bijection from the 
set of Jordan types realized as Jordan types of $kG$-modules of 
constant Jordan type to the set of Jordan types realized as Jordan 
types of $kH$-modules of constant Jordan type?
\end{ques}


\section{APPENDIX: Decomposition of tensor products of $k[t]/t^p$-modules}

The purpose of this appendix is to derive a closed
form for tensor products of $k[t]/t^p$-modules, presumably implicit
in \cite{Srin}.  Here, we view
$k[t]/t^p$ as a self-dual Hopf algebra; in other words, as the 
restricted enveloping algebra of the 1-dimensional $p$-restricted
Lie algebra (with trivial $p$-restriction operator). 
Thus, the coproduct is given by the formula 
$t \mapsto 1 \otimes t + t \otimes 1$.   

\vspace{0.1in}
Let $V(\lambda)$ denote the simple restricted $sl_2$-module of highest weight 
$\lambda$ where $0 \leq \lambda \leq p-1$.
Let $e, h, f$ be the standard generators for $sl_2$. 
Let $\langle e \rangle \subset sl_2$ be the 1-dimensional $p$-restricted 
Lie algebra generated by the element $e$.  Hence,   
$u(\langle e \rangle) \simeq  k[t]/t^p$ is a Hopf subalgebra of $u(sl_2)$.

The $sl_2$-representation theory implies that
\begin{equation}
\label{restriction}
V(\lambda)\downarrow_{u(\langle e \rangle)} \simeq [\lambda + 1]
\end{equation} 
as $k[t]/t^p$-modules.  
Hence, the tensor product formulas for indecomposable $k[t]/t^p$ 
modules follow from the tensor product formulas for simple $sl_2$-modules.   
Such formulas were obtained  in \cite{Pre} (see also \cite{BO}).    

\begin{prop} (See \cite[Lemma 2.3, 2.4]{Pre}.) 
Let $\{V(\lambda)\}_{0\leq \lambda \leq p-1}$ be the collection of 
simple restricted $sl_2$-modules.

\vspace{0.1in}

\noindent  
$\bu$ If  $0 \leq \mu\leq \lambda \leq p-2$ and $\lambda + \mu \leq p-1$  then 
$$ 
V(\lambda) \otimes V(\mu) \simeq V(\lambda -\mu)\oplus 
V(\lambda - \mu + 2) \oplus \ldots \oplus V(\lambda+\mu)$$
$\bu$ If  $0 \leq \mu\leq \lambda \leq p-2$ and $\lambda + 
\mu \geq p$  then 
$$ 
V(\lambda) \otimes V(\mu) \simeq  V(\lambda -\mu)\oplus 
V(\lambda - \mu + 2)\oplus \ldots \oplus  V(2(p-2) - \lambda - \mu) + \proj
$$
where the projective summand  has dimension $\lambda + \mu - (p-2)$.

\end{prop}

Restricting the tensor decompositions for the simple 
restricted $sl_2$-modules $V(\lambda)$ to the indecomposable 
$k[t]/t^p$-modules $[i]$ via the formula  (\ref{restriction}) 
we obtain the following formulas.

\begin{cor}
\label{tensor2}
Let $[i]$  be an indecomposable  $k[t]/t^p$-module of dimension 
$i$ for $1\leq i \leq p$. Then
\begin{equation*}
[i]\otimes[j] =  \left\{\begin{array}{ll}
{[j-i +1] + [j-i + 3] + \ldots + [j+i-3] + [j+i-1]} &
 \text{ if \,\,} j+i \leq p  \\
 &\\
 {[j -i +1] 
 + \ldots  + [2p-1-i-j]   + (j+i -p)  [p] } & \text{ if \,\,} j+i > p.
 \end{array} \right.
\end{equation*}
 \end{cor}

We conclude this appendix with the analogous formula 
for the tensor product of indecomposable $kC_p$-modules
where $C_p$ is a cyclic group of order $p$. A subtlety 
here is that even though the module categories for $kC_p$ 
and the algebra $k[t]/t^p$ with the coproduct 
$ t \mapsto 1 \otimes t + t \otimes 1$ are equivalent, 
the tensor product structure comes from two 
different coproducts. Nonetheless, the tensor multiplicities 
turn out to be the same. 
\begin{cor}
Let $[i]$, $1 \leq i \leq p$, be indecomposable $kC_p$-modules.   Then 
the decomposition of the tensor product 
$[i] \otimes [j]$  into indecomposable $kC_p$-modules 
is given by the formulas as in Corollary 
\ref{tensor2}.
\end{cor}

\begin{proof}
By \cite[4.5]{FPS} the tensor product of any two $kC_p$-modules 
$M$, $N$ is isomorphic as 
$k[t]/t^p$-module to the tensor product $M \otimes N$  
using the coproduct $t \mapsto 1 \otimes t + t \otimes 1$.
The statement now follows from Corollary \ref{tensor2}.
\end{proof}


\end{document}